\documentclass[12pt]{amsart}
\usepackage[utf8]{inputenc}
\usepackage{amsthm}
\usepackage{amsmath}
\usepackage{amssymb}
\usepackage{comment}
\usepackage[margin=1.2in]{geometry}

\usepackage{times}
\usepackage[T1]{fontenc}
\usepackage{enumitem}
\usepackage{setspace}
\usepackage{microtype}
\usepackage{cite}
\usepackage{xcolor}

\allowdisplaybreaks
\numberwithin{equation}{section}

\usepackage[colorlinks=true, pdfstartview=FitV, linkcolor=blue, citecolor=blue, urlcolor=blue]{hyperref}

\newtheorem{theorem}{Theorem}[section]
\newtheorem{definition}[theorem]{Definition}
\newtheorem{lemma}[theorem]{Lemma}
\newtheorem{prop}[theorem]{Proposition}
\newtheorem{corollary}[theorem]{Corollary}
\newtheorem{remark}[theorem]{Remark}
\newtheorem{example}[theorem]{Example}

\title{The Ross-Darboux-Stieltjes Integral}

\author{David Cruz-Uribe, OFS}
\address{Dept. of Mathematics \\
University of Alabama \\
 Tuscaloosa, AL 35487, USA}
\email{dcruzuribe@ua.edu}

\author{Jacob Glidewell}
\address{Dept. of Mathematics \\
University of Alabama \\
 Tuscaloosa, AL 35487, USA}
\email{jbglidewell@crimson.ua.edu}

\begin{document}

\begin{abstract}
    Motivated by the limitations of the traditional definitions of the Riemann-Stieltjes and  Darboux-Stieltjes integrals, we introduce a generalized Darboux-Stieltjes integral that is equivalent to an earlier generalization by Ross~\cite{Ross}. Our definition builds upon an approach to the Darboux-Stieltjes integral recently introduced by the first author and Convertito~\cite{TSI}.  We show that our definition agrees with all previous definitions, but that the class of integrable functions is much  larger.  We develop all the analogs of the classic results for the Riemann integral,  and rectify the problems inherent in the definition of the  Darboux-Stieltjes integral in~\cite{TSI}. In particular, we show the Bounded Convergence Theorem holds for our definition and that it agrees with the Lebesgue-Stieltjes integral.
\end{abstract}

\date{May 27, 2024}

\keywords{Stieltjes integral, Lebesgue-Stieltjes integral, functions of bounded variation}

\subjclass[2010]{Primary 26A42}

\thanks{This paper is a revised version of the masters thesis written by the second author under the direction of the first.  The first author is partially supported by a Simons Foundation Travel Support for Mathematicians Grant. }
\maketitle
\tableofcontents

\section{Introduction}
\label{section:intro}

The Stieltjes integral plays an important role in analysis and in applications.  However, since its introduction it has been recognized that its original definition had serious drawbacks, and there have been several attempts to improve it, none completely successful.  The goal of this paper is to introduce a new definition which we believe overcomes the drawbacks of previous definitions.  To support this claim and to provide context for our results, we will first review the history of the Stieltjes integal. This review is not intended to be comprehensive, but should suffice to illustrate our point.  Hereafter, $B[a,b]$ will be the set of bounded functions on the interval $[a,b]$, $I[a,b]$, will be increasing functions, and $BV[a,b]$ will be functions of bounded variation.  

Stieltjes \cite{Stieltjes} first introduced the Riemann-Stieltjes integral by generalizing the definition of the Riemann integral.

\begin{definition} \label{defn:rs-integral}
    Given $f\in B[a,b]$, $\alpha\in BV[a,b]$, a partition $\mathcal P = \{x_i\}_{i=0}^n$ of $[a,b]$ and sample points $x_i^*\in [x_{i-1}, x_i]$, we form a Riemann-Stieltjes sum 
    $$\sum_{i=1}^n f(x_i^*)[\alpha(x_i)-\alpha(x_{i-1})].$$ 
    If there exists $A$ such that for every $\epsilon>0$, there exists $\delta>0$ such that for every Riemann-Stieltjes sum with $|\mathcal P|<\delta$, $$\left|\sum_{i=1}^n f(x_i^*)[\alpha(x_i)-\alpha(x_{i-1})] - A\right|<\epsilon,$$ then we define the Riemann-Stieltjes integral of $f$ with respect to $\alpha$ by
    $$(RS)\int_a^b f(x)\, d \alpha = A.$$ 
\end{definition}

Stieltjes introduced this integral, hereafter the RS-integral, to study continued fractions.  Later, M.~ Riesz~\cite{Riesz} showed the importance of this definition by using the RS-integral to characterize the space of linear functionals on $C[a,b]$, the space of continuous functions on $[a,b]$. (See also F.~Riesz and Sz.-Nagy~\cite{MR0071727}.) Both Stieltjes and Riesz were primarily interested in this integral when the integrand $f$ was a continuous function.

Straightforward examples (see~\cite[Example~5.30]{TSI}) show that the RS-integral only exists if $f$ and $\alpha$ have no common points of discontinuity. To overcome this weakness and widen the class of functions for which a Stieltjes-type integral is defined, Pollard~\cite{Pollard} gave two equivalent definitions: the Pollard-Stieltjes integral (PS-integral) and what is now referred to as the Darboux-Stieltjes integral (DS-integral). The PS-integral changes Definition~\ref{defn:rs-integral} by  replacing the  convergence of the Riemann-Stieltjes sums (referred to as "norm convergence" in~\cite{MR1524276}) with a weaker form of convergence (referred to as "$\sigma$-convergence" in~\cite{MR1524276}).  More precisely,  for each $\epsilon>0$, it only considers the Riemann-Stieltjes sums on the refinements of a specific partition that depends on $\epsilon$, $\mathcal{P}_\epsilon$,  instead of all partitions with $|\mathcal P|<\delta$. (See Definition~\ref{def:Pollardint} for more details.) 

The DS-integral generalizes the definition of the classical Darboux integral (often mistakenly referred to as the Riemann integral). Pollard only stated it for $\alpha$ increasing; however, it can be extended to all $\alpha\in BV[a,b]$ using linearity and the Jordan decomposition theorem.

\begin{definition} \label{defn:DS-integral}
    Given $f\in B[a,b]$, $\alpha\in I[a,b]$, and a partition $\mathcal P = \{x_i\}_{i=1}^n$, define 
    $$\overline{M}_i = \sup\{f(x)\; : x\in [x_{i-1}, x_i]\}, 
    \quad 
    \overline{m}_i = \inf\{f(x)\; : x\in [x_{i-1}, x_i]\},$$ 
    and form the upper and lower sums 
    \begin{align*}
        U_\alpha^e (f, \mathcal{P}) 
        &= \sum_{i=1}^n\overline{M}_i[\alpha(x_i)-\alpha(x_{i-1})],\\
        L_\alpha^e (f, \mathcal{P}) 
        &= \sum_{i=1}^n\overline{m}_i[\alpha(x_i)-\alpha(x_{i-1})].
    \end{align*}
    If $\inf U_\alpha^e (f, \mathcal{P})=\sup L_\alpha^e (f, \mathcal{P})=A$, where the infimum and supremum are taken over all partitions of $[a,b]$, then we define the Darboux-Stieltjes integral of $f$ with respect to $\alpha$ by
    $$(DS)\int_a^b f(x)\, d \alpha=A.$$ 
\end{definition}

For $\alpha\in I[a,b]$, Pollard showed the PS-integral and the DS-integral are equivalent and agree with the RS-integral whenever the latter exists.  However, both of these integrals  are  more general than the RS-integral: there exist functions which are DS-integrable  but whose RS-integrals do not exist since $f$ and $\alpha$ have common discontinuities.  The correct necessary condition in this case (see \cite[Theorem~4.13]{Nielson}) is that the DS-integral exists if at each common discontinuity $x\in [a,b]$, at least one of $f$ or $\alpha$ is right continuous and at least one is left continuous at $x$. 

Several other versions of the Stieltjes integral have appeared in the literature. Hildebrandt~\cite{MR1524276} and Dushnik~\cite{MR2936614} studied generalizations of the PS-integral.  Pollard gave an ``interior" version of the DS-integral where $\overline{M}_i$ and $\overline{m}_i$ are replaced with $M_i$ and $m_i$, the supremum and infimum of $f(x)$ when $x\in (x_{i-1}, x_i)$. To distinguish this from Definition~\ref{defn:DS-integral} we will sometimes refer to that one as the "exterior" DS-integral.  (This explains the superscript "$e$" that appears above.)  An equivalent version of the interior DS-integral was independently discovered by Convertito and the first author and developed in~\cite{TSI}.  It is equivalent to the exterior DS-integral wherever the the latter is defined, but is defined on a larger collection of functions:  in~\cite[Theorem~5.19]{TSI} they showed that a necessary condition for the integral to exist is that at  each common discontinuity $x\in [a,b]$, $\alpha$ is right continuous or $f(x+)$ exists, and $\alpha$ is left continuous or $f(x-)$ exists.

The RS, DS, and PS-integrals were the definitions used for the Stieltjes integral in many of the standard analysis books in the second half of the twentieth century.  The RS-integral was used by  Taylor and Mann~\cite{MR0674807}. Apostol~\cite{MR0087718} used the PS-integral and showed the equivalence of the DS-integral as theorem.  He noted in an exercise that the RS-integral is not equivalent.     Bartle~\cite{MR0393369} also used the PS-integral and gave the RS and DS-integrals as exercises.  From a pedagogic perspective, we note that neither author explains the rationale for using $\sigma$-convergence in their definition.  Rudin~\cite{MR0385023} used the DS-integral, as did Burkhill and Burkhill~\cite{MR1962361}, and Krantz~\cite{MR3617044}.  Protter and Morrey~\cite{MR0463372} and Wheeden and Zygmund~\cite{MR3381284} defined both the RS and DS-integrals and highlighted the fact that they are not equivalent.

Ross~\cite{Ross,MR0600928}, motivated by the restrictive necessary conditions for the RS and DS-integrals to exist, proposed a modified version of the DS-integral.

\begin{definition}\label{def: RossIntegral}
     Given $f\in B[a,b]$, $\alpha\in I[a,b]$, and a partition $\mathcal P = \{x_i\}_{i=1}^n$, define 
     $$M_i = \sup\{f(x)\; : x\in (x_{i-1}, x_i)\}, 
     \quad m_i = \inf\{f(x)\; : x\in (x_{i-1}, x_i)\},$$ 
     and 
     $$J_\alpha(f, \mathcal{P})= \sum_{i=0}^n f(x_i)[\alpha(x_i+)- \alpha(x_i-)].$$ 
     Form the upper and lower sums 
     \begin{align*}
        U_\alpha^r (f, \mathcal{P}) 
        &= J_\alpha(f, \mathcal{P})+\sum_{i=1}^n M_i[\alpha(x_i-)-\alpha(x_{i-1}+)],\\
        L_\alpha^r (f, \mathcal{P}) 
        &= J_\alpha(f, \mathcal{P})+\sum_{i=1}^n m_i[\alpha(x_i-)-\alpha(x_{i-1}+)].
    \end{align*}
    If $\inf U_\alpha^r (f, \mathcal{P})=\sup L_\alpha^r (f, \mathcal{P})=A$, where the infimum and supremum are taken over all partitions of $[a,b]$, then we define the Ross-Darboux-Stieljtes (RDS) integral of $f$ with respect to $\alpha$ by
    $$(RDS)\int_a^b f(x)\, d \alpha=A.$$ 
\end{definition}

Ross~\cite{Ross} also gave an equivalent definition in terms of a modified RS-integral.  He loosely motivated the additional term $J_\alpha$ used to define the upper and lower sums by suggesting that it allows individual points to be assigned a positive "weight."  It is clear from other comments in~\cite{MR0600928} that he was also motivated by the Lebesgue-Stieltjes integral.  

\begin{remark}
Another variant of the RS-integral was proposed by Ter Horst~\cite{LebStieltjes} which is a generalization of the definition given by Ross.  His definition is a hybrid between the classical Stieltjes integrals and the more complex definition of the Lebesgue integral.  We will not discuss it in detail in this paper.
\end{remark}

\medskip
As noted above, in~\cite{TSI} the authors developed the interior DS-integral using an equivalent definition in terms of approximation by step functions. (See Definition~\ref{def:DSint} below.)  Like Ross, they wanted to weaken the restrictive necessary conditions associated with the standard definitions.  They were also motivated by pedagogical concerns:  they felt that their definition was intuitive and also created a bridge between classical undergraduate analysis and graduate analysis and measure theory.  However, as they pointed out in their introduction, there are two significant problems with the interior DS-integral. First, the Bounded Convergence Theorem (the analog of the Lebesgue Dominated Convergence Theorem for the DS-integral) does not hold for this definition of the Stieltjes integral unless $\alpha$ is continuous. See~\cite[Example~4.49 and Theorem~4.50]{TSI}.   Second, and perhaps more problematic from an advanced perspective, this definition does not always  agree with the Lebesgue-Stieltjes integral when $f$ and $\alpha$ have a common discontinuity. 

\medskip

The goal of this paper is to address the problems with the interior DS-integral as defined in~\cite{TSI}.  We give  a definition which is equivalent to Definition~\ref{def: RossIntegral} but which adopts the perspective taken in~\cite{TSI}.  In particular, we draw a more explicit connection to the   Lebesgue-Stieltjes integral.  In the definitions of the RS, PS and DS-integrals, as well as in Definition~\ref{def:DSstepint} below, the "$\alpha$-length" of a partition interval $[x_{i-1},x_i]$ is given by $\alpha(x_i)-\alpha(x_{i-1})$.  The problem with this, as Ross correctly noted, is that the Stieltjes integral should give weight to individual points in the partition if they correspond to jump discontinuities of $\alpha$.  To accomplish this, we will replace this definition of $\alpha$-length with the following:  given an open interval $I=(c,d)$, we define $\alpha(I)=\alpha(d-)-\alpha(c+)$, where $x\pm$ denotes the  right and left-sided limits at $x$.  We also define the $\alpha$-length of a single point $x$ by $\alpha(\{x\}) = \alpha(x+)-\alpha(x-)$.  This modest change lets us correctly capture all the behavior of the Stieltjes integral when $f$ and $\alpha$ have common discontinuities.  This notion of $\alpha$-length is implicit in the definition of $J_\alpha$ in Definition~\ref{def: RossIntegral}; it also corresponds to the premeasure used to define the Lebesgue-Stieltjes measure. (See \cite[Section~4.1]{LebesgueStieltjesCarter}.)

Given our definition, we show that all of the standard properties of the Stieltjes integral can be easily derived.  These results can be thought of as extending and complementing the development of the interior DS-integral in~\cite{TSI} and the development of the RDS-integral in~\cite{Ross}.  We then prove some more advanced theorems which show that in some sense the RDS-integral is the right definition to use.  We completely characterize the functions which are RDS-integrable by proving the precise analogue of the Lebesgue criterion for Riemann-Darboux integrability. 

\begin{theorem}[RDS-Lebesgue Criterion]
Given $f\in B[a,b]$ and $\alpha\in BV[a,b]$, $f\in RDS_\alpha[a,b]$ if and only if the subset of $[a,b]$ where $f$ is discontinuous has $\alpha$-measure $0$.
\end{theorem}

As an immediate consequence of this theorem we have that if $f$ and $\alpha$ are functions of bounded variation, then the RDS-integral of $f$ with respect to $\alpha$ exists.  
We also prove that the Bounded Convergence theorem holds for the RDS-integral.

\begin{theorem}[RDS Bounded Convergence Theorem]
Given $\alpha \in BV[a,b]$ and a uniformly bounded sequence of  functions $\{f_n\}_{n=1}^\infty$ on $[a,b]$, if 

\begin{itemize}
    \item for all $n\in \mathbb{N}$, $f_n\in RDS_\alpha[a,b]$,
    
    \item $f_n\to f$ pointwise as $n\to \infty$,
    
    \item $f\in RDS_\alpha[a,b]$,
\end{itemize}
then $$\lim_{n\to\infty}\int_a^b f_n(x)\, d \alpha = \int_a^b f(x)\, d \alpha.$$
\end{theorem}

Finally, we prove that the RDS-integral agrees with the Lebesgue-Stieltjes integral whenever the former is defined.

\medskip

The remainder of this paper is organized as follows. In Section~\ref{section:prelim} we give some preliminary definitions and sketch  the approach in \cite{TSI} to  the Darboux-Stieltjes integral.  This will serve as a foundation for our definition of the RDS-integral, which we give in Sections~\ref{section:rds-defn} and~\ref{section:rds-defn-bis}. In Section~\ref{section:rds-defn-bis} we also prove some basic properties of the RDS-integral, including the RDS criterion, an analogue of the Darboux criterion for integrability. Finally, we give sufficient conditions, in terms of continuity, for the RDS and DS-integrals to agree.   In Section~\ref{section:RDS-criterion}, we use the  RDS criterion to prove  linearity, additivity, monotonicity in the integrand and linearity of the integrator. In Section~\ref{section:convergence}, we prove two convergence theorems: uniform convergence in the integrand and BV-norm convergence in the integrator. In Section~\ref{section:integrability} we consider necessary and sufficient conditions for RDS-integrability and prove the RDS-Lebesgue criterion. As an application we prove a generalization of integration by parts for the RDS-integral.  In Section~\ref{section:BCT} we prove the Bounded Convergence Theorem. Finally, in Section~\ref{section:alt-defns}, we compare various definitions of the Stieltjes integral.  We first give a characterization the functions whose RDS and DS-integrals agree.  We then give  three  alternative definitions of a generalized Riemann-Stieltjes integral and discuss their relationship with the RDS-integral.  Finally, we prove that the RDS-integral and the Lebesgue-Stieltjes integral agree.

Throughout, our notation will follow that used in~\cite{TSI}; similarly, in our proofs we will draw upon a number of results proved in this work.  In Sections~\ref{section:rds-defn}--\ref{section:convergence} many of our proofs are very similar to proofs of the corresponding result for the Darboux or DS-integrals, and we have tried to strike a balance between being self-contained and repeating arguments that can be found elsewhere.  The results in Sections~\ref{section:integrability}--\ref{section:alt-defns} are (with one exception) all new, and we give complete proofs.

\section{The  Darboux-Stieltjes Integral}
\label{section:prelim}

In this section we briefly state the definition of the interior Darboux-Stieltjes integral as defined in~\cite[Chapter~4]{TSI}.  This will help to motivate the definition of the RDS-integral that we will give in Section~\ref{section:rds-defn} below.  As part of of this we define functions of bounded variation; additional properties of functions of bounded variation will be called upon throughout the text, and we refer to~\cite[Chapter~3]{TSI} for complete information.

Fundamental to all of our definitions is the notion of a partition of an interval.

\begin{definition}\label{def:part}
Given a closed interval $[a,b]$, a partition of $[a,b]$ is a finite set of points $\mathcal{P}=\{x_i\}_{i=0}^n$ such that $$a=x_0<x_1<x_2<\dots<x_{n-1}<x_n=b.$$ Given $\mathcal{P}$, define the disjoint partition intervals $I_i=(x_{i-1},x_i)$ for each $1\le i\le n$. The mesh size $|\mathcal{P}|$ of a partition is the length of the longest partition interval:  
\[ |\mathcal{P}|= \max_{1\leq i \leq n} \{ x_i-x_{i-1}\}. \]
\end{definition}

\begin{definition}\label{def:refine}
Given partitions $\mathcal{P}=\{x_i\}_{i=0}^n$ and $\mathcal{Q}=\{y_j\}_{j=0}^m$ of the interval $[a,b]$, we say $\mathcal{Q}$ is a refinement of $\mathcal{P}$ if $\mathcal{P}\subset \mathcal{Q}$. We say $\mathcal{R}$ is a common refinement of $\mathcal{P}$ and $\mathcal{Q}$ if $\mathcal{R}$ is a refinement of both $\mathcal{P}$ and $\mathcal{Q}$.
\end{definition}

We will use the above notation throughout. Moreover, given a second partition $\mathcal{Q}=\{y_j\}_{j=0}^m$, we will use $J_j=(y_{j-1}, y_j)$ for each $1\le j\le m$ for the partition intervals of $\mathcal{Q}$.

\medskip

We can now define the bounded variation of a function.

\begin{definition} \label{defn:BV}
Given a function $f\in B[a,b]$ and a partition $\mathcal{P}$ of $[a,b]$, define the variation of $f$ on $\mathcal{P}$ by
\[ V(f,\mathcal{P}) = \sum_{i=0}^n |f(x_i)-f(x_{i-1})|. 
\]
We say that $f$ is of bounded variation, and write $f\in BV[a,b]$, if 
\[ V(f,[a,b]) = \sup_{\mathcal{P}} V(f,\mathcal{P}) < \infty,
\]
where the supremum is taken over all partitions $\mathcal{P}$ of $[a,b]$.
\end{definition}

For the definition of the DS-integral, we recall the Jordan decomposition theorem:  if $\alpha\in BV[a,b]$, then there exist functions $\alpha^+,\,\alpha^- \in I[a,b]$ such that $\alpha=\alpha^+-\alpha^-$.  

\medskip

The definition of the DS-integral occurs in two stages.  First, we define it for step functions.  The integrals of step functions will take the place of the upper and lower Darboux sums in the traditional definition, e.g., Definition~\ref{defn:DS-integral}.

\begin{definition}\label{def:stepf}
A function $f\in B[a,b]$ is a step function if there exists a partition $\mathcal{P}$ of $[a, b]$ and real numbers $\{c_i\}_{i=1}^n$ such that $f(x) = c_i$ if $x \in I_i, 1 \le i \le n$. The function f is said to be a step function with respect to the partition $\mathcal{P}$. 
\end{definition}

\begin{definition}\label{def:wrongalphalength}
    Given $\alpha\in BV[a,b]$ and an open interval $I=(c,d)\subset [a,b]$, we define the (classical) $\alpha$-length of $I$ to be $\alpha(I)=\alpha(d)-\alpha(c)$.
\end{definition}

We refer to this as the classical $\alpha$-length because of the role it played historically in the definition of the Stieltjes integral.  Below in  Definition~\ref{def:aleng} we will redefine it to correspond to the premeasure used to define the Lebesgue-Stieltjes integral.

\begin{definition}\label{def:DSstepint}
    Given $f\in S[a,b]$ defined with respect to a partition $\mathcal{P}$, suppose $f(x)=c_i$ for $x\in I_i$. Define the Stieltjes integral with respect to $\alpha\in BV[a,b]$ by $$(DS)\int_{a}^{b}f(x)\, d \alpha = \sum_{i=1}^nc_i\alpha(I_i).$$
\end{definition}

As the second step, we define the DS-integral of any bounded function.  To do so, we define the collections of step functions that bracket $f$.

\begin{definition}\label{def:bracket}
Given $f\in B[a,b]$, define the sets 
\begin{align*}
    S_U(f, [a,b])&= \{u\in S[a,b]\; :\; f(x)\le u(x), x\in[a,b]\},\\
    S_L(f, [a,b])&= \{v\in S[a,b]\; :\; f(x)\ge v(x), x\in[a,b]\}.
\end{align*}

We say $u,\,v\in S[a,b]$ bracket $f$ if $u\in S_U(f, [a,b])$ and $v\in S_L(f, [a,b]).$
\end{definition}

\begin{definition}\label{def:DSint}
    Given $\alpha\in I[a,b]$ and $f\in B[a,b]$, define
\begin{align*}
    L^i_\alpha (f, [a,b])=\sup\left\{\int_{a}^b v(x)\, d \alpha \; :\; v\in S_L(f, [a,b])\right\},\\
    U^i_\alpha (f, [a,b])=\inf\left\{\int_{a}^b u(x)\, d \alpha \; :\; u\in S_U(f, [a,b])\right\}.
\end{align*}

If $L^i_\alpha(f, [a,b])=U^i_\alpha(f, [a,b])$, then we say $f$ is (interior) Darboux-Stieltjes integrable with respect to $\alpha$ and write $$(DS)\int_a^bf(x)\, d \alpha$$ as their common value. 
Moreover, for $\alpha\in BV[a,b]$, if there exists a decomposition $\alpha=\alpha^+ - \alpha^-$ where $f\in DS_{\alpha^+}[a,b]$ and $f\in DS_{\alpha^-}[a,b]$, then we define $f\in DS_\alpha[a,b]$ and $$(DS)\int_a^bf(x)\, d \alpha=(DS)\int_a^bf(x)\, d \alpha^+-(DS)\int_a^bf(x)\, d \alpha^-.$$
We denote the set of functions which are Darboux-Stieltjes integrable with respect to $\alpha$ as $DS_\alpha[a,b]$.
\end{definition}

The Darboux-Stieltjes integral is  well-defined; in particular, Definition~\ref{def:DSint} applied to a step function agrees with Definition~\ref{def:DSstepint}.   Moreover, this definition is equivalent to the definition of the interior DS-integral described in Section~\ref{section:intro}: details are left to the interested reader.  See the proof of the corresponding result for the RDS-integral, Corollary~\ref{cor:bestfitRDScrit}.

\section{The Ross-Darboux-Stieltjes Integral for Step functions}
\label{section:rds-defn}

We now  turn to the definition of  the RDS-integral.   Following the pattern of the definition of the DS-integral described in Section~\ref{section:prelim}, in this section we give a new definition of $\alpha$-length and define the RDS-integral of step functions.  We will define the integral in general, and prove some of its basic properties, in Section~\ref{section:rds-defn-bis}.

We first give our revised definition of $\alpha$-length, which takes more careful account of the jump discontinuities of the integrator $\alpha$.  

\begin{definition}\label{def:aleng}
Define the $\alpha$-length of an open interval $I=(c,d)$ as $$\mu_{\alpha}(I)=\alpha(d-)-\alpha(c+),$$ of a closed interval $\overline{I}=[c,d]$ as $$\mu_{\alpha}(\;\overline{I}\;)=\alpha(d+)-\alpha(c-),$$ and of a singleton set $\{c\}$ as $$\mu_{\alpha}(\{c\})=\alpha(c+)-\alpha(c-).$$ We also extend this to half-open intervals 
$$\mu_\alpha ((c,d])= \alpha (d+)-\alpha(c+) \quad \text{and} \quad \mu_\alpha([c,d))= \alpha(d-)-\alpha(c-).$$
\end{definition}

Clearly,  $\mu_\alpha$ is finitely additive. Moreover, if $\alpha$ is increasing, $\mu_\alpha$ is nonnegative.

\begin{remark}
    In dealing with a function $\alpha$ defined on a closed interval $[a,b]$, we adopt the following conventions: $\alpha(a-)=\alpha(a)$ and $\alpha(b+)=\alpha(b)$.  This convention will be important in the proving the additivity of the RDS-integral. See the proofs of Theorems~\ref{thm:stepadd} and~\ref{thm:additivity}.
\end{remark}

We now  define the integral of a step function. Notice the definition has two terms: the first is analogous to Ross's jump term $J_\alpha(f, \mathcal{P})$ in Definition \ref{def: RossIntegral}, and the second is the classical step function integral as in Definition \ref{def:DSstepint}.

\begin{definition}\label{def:stepint}
Given $f\in S[a,b]$ defined with respect to a partition $\mathcal{P}$, suppose $f(x)=c_i$ for $x\in I_i$. Define the Ross-Stieltjes integral with respect to $\alpha\in BV[a,b]$ by $$(\mathcal{P})\int_{a}^{b}f(x)\, d \alpha = \sum_{i=0}^n f(x_i)\mu_{\alpha}(\{x_i\})+\sum_{i=1}^nc_i\mu_{\alpha}(I_i).$$
\end{definition}

Clearly, a step function can be defined with respect to multiple paritions:  for instance, a constant function is a step function with respect to any partition.  Therefore, we have to show our definition is well-defined, i.e., that it  is independent of the chosen partition.

\begin{prop}\label{prop:partinvariant}
Given $\alpha \in BV[a,b]$, if $f\in S[a,b]$ is defined with respect to both partitions $\mathcal{P}$ and $\mathcal{Q}$, then $$(\mathcal{P})\int_{a}^{b}f(x)\, d \alpha =(\mathcal{Q})\int_{a}^{b}f(x)\, d \alpha.$$
\end{prop}

\begin{proof}

First, suppose $\mathcal{Q}=\{y_j\}_{j=0}^m$ is a refinement of $\mathcal{P}=\{x_i\}_{i=0}^n$ with only one additional point. That is, $m=n+1$.
Note there exists indices $1\le \ell \le n$, $1\le k\le n$ such that $$y_{k-1}=x_{\ell-1} < y_k < y_{k+1}=x_{\ell}.$$ Hence, $f(y_k)=c_\ell = d_k=d_{k+1}$. Then, 

\begin{align*}
 &   (\mathcal{Q})\int_{a}^{b}f(x)\, d \alpha \\
 & \qquad =\sum_{j=0}^{n+1} f(y_j)\mu_{\alpha}(\{y_j\})+\sum_{j=1}^{n+1} d_j\mu_{\alpha}(J_j)\\
    & \qquad = f(y_k)\mu_\alpha(\{y_k\}) + \sum_{i=0}^n f(x_i)\mu_{\alpha}(\{x_i\})\\
    & \qquad \quad +d_k\mu_\alpha(J_k)+d_{k+1}\mu_\alpha(J_{k+1})-c_{\ell}\mu_\alpha(I_\ell)+\sum_{i=1}^n c_i\mu_{\alpha}(I_i) \\ 
    & \qquad = f(y_k)\mu_\alpha(\{y_k\})+\sum_{i=0}^n f(x_i)\mu_{\alpha}(\{x_i\})\\
    & \qquad \quad +c_\ell [\alpha(y_k-)-\alpha(y_{k-1}+)+\alpha(y_{k+1}-)-\alpha(y_k+)-\alpha(x_\ell-)+\alpha(x_{\ell-1}+)]\\
    & \qquad \quad +\sum_{i=1}^n c_i\mu_{\alpha}(I_i) \\
    & \qquad =f(y_k)\mu_\alpha(\{y_k\})+c_\ell [\alpha(y_k-)-\alpha(y_k+)]+\sum_{i=0}^n f(x_i)\mu_{\alpha}(\{x_i\})+\sum_{i=1}^n c_i\mu_{\alpha}(I_i)\\
    & \qquad =f(y_k)\mu_\alpha(\{y_k\})-f(y_k)\mu_\alpha(\{y_k\})+\sum_{i=0}^n f(x_i)\mu_{\alpha}(\{x_i\})+\sum_{i=1}^n c_i\mu_{\alpha}(I_i)\\
    & \qquad =\sum_{i=0}^n f(x_i)\mu_{\alpha}(\{x_i\})+\sum_{i=1}^n c_i\mu_{\alpha}(I_i)\\
    & \qquad =(\mathcal{P})\int_{a}^{b}f(x)\, d \alpha.
\end{align*}

Next, assume $\mathcal{Q}$ is a refinement of $\mathcal{P}$. Thus, we can write $\mathcal{Q}=\mathcal{P}\cup\{t_1, \dots, t_r\}$. Define $\mathcal{P}_0=\mathcal{P}$, $\mathcal{P}_1=\mathcal{P}_0\cup\{t_1\}, \mathcal{P}_2=\mathcal{P}_1\cup\{t_2\}, \dots, \mathcal{P}_r=\mathcal{P}_{r-1}\cup\{t_r\}=\mathcal{Q} $. From the previous case, for each $0\le i\le r-1$, $$(\mathcal{P}_i)\int_{a}^{b}f(x)\, d \alpha=(\mathcal{P}_{i+1})\int_{a}^{b}f(x)\, d \alpha.$$ By combining these equalities, we have $$(\mathcal{P})\int_{a}^{b}f(x)\, d \alpha=(\mathcal{P}_{0})\int_{a}^{b}f(x)\, d \alpha=(\mathcal{P}_{r})\int_{a}^{b}f(x)\, d \alpha=(\mathcal{Q})\int_{a}^{b}f(x)\, d \alpha.$$

Finally, we prove the general case. Let $\mathcal{R}$ be a common refinement of $\mathcal{P}$ and $\mathcal{Q}$. Thus, by the previous case, $$(\mathcal{P})\int_a^b f(x)\, d \alpha = (\mathcal{R})\int_a^b f(x)\, d \alpha= (\mathcal{Q})\int_a^b f(x)\, d \alpha. $$
\end{proof}

Given Proposition \ref{prop:partinvariant}, we may drop the indication of the partition and simply write 
\[ (S) \int_a^b f(x)\,d\alpha \]
for the RDS integral of a step function.  In fact, for the rest of this section, for brevity we will drop the "$(S)$" as well.  We will use this notation again in Section~\ref{section:rds-defn-bis}.  

\medskip

The canonical example of a step function is the Heaviside function; we will use this function repeatedly throughout this paper. 

\begin{example}\label{lem:H_c}
Given $c\in \mathbb{R}$, let $$H_c(x)=\begin{cases}0, & x<0, \\ c, & x=0,\\ 1, & x> 0.\end{cases}
$$ Then, for every $a, b\in\mathbb{R}$, 
$$\displaystyle (S) \int_{-1}^1 H_a(x) \, d H_{b} = a.$$ However, 
$$(DS)\displaystyle \int_{-1}^1 H_a(x) \, d H_{b} = 1-b.$$
\end{example}

\begin{remark} 
In~\cite{TSI}, the right continuous function $H_1$ is referred to as "the" Heaviside function, and the left continuous function $H_0$ is referred to as the "Jeaviside" function.
\end{remark}

\begin{proof}
We use the partition $\mathcal{P}=\{-1, 0, 1\}$. Note $H_c$ is a step function with respect to this partition. Now,

\begin{align*}
    \displaystyle \int_{-1}^1 H_a(x) \, d H_b &= H_a(-1)\mu_{H_b} (\{-1\})+H_a(0)\mu_{H_b} (\{0\})+H_a(1)\mu_{H_b} (\{1\})\\&\quad +(0)\mu_{H_b} ((-1,0))+(1)\mu_{H_b} ((0,1))\\ \\ &=(0)(0)+(a)(1)+(1)(0)+(0)(0)+(1)(0) \\ &=a.
\end{align*} However, \begin{align*}
    (DS)\displaystyle \int_{-1}^1 H_a(x) \, d H_b &=(0)H_b ((-1,0))+(1)H_b ((0,1))\\ &=(0)(b-0)+(1)(1-b) \\ &=1-b.
\end{align*}
\end{proof}

In Example~\ref{lem:H_c}, $a$ is the desired value: since  $H_b$ is flat except at zero where it jumps, the integral should only "see" $H_a$ at zero and be equal to its value there.  Moreover, this agrees with the Lebesgue-Stieltjes integral, since this is what is gotten by integrating the function $H_a$ against a point mass at the origin (modeled by integrating with respect to $H_b)$.  The DS-integral, however, only correctly models integration against a point mass  when the integrand is continuous:  see~\cite[Lemma~5.6]{TSI}.  

\medskip

Another application of Definition~\ref{def:stepint} we will continually use is the integral of a constant. 

\begin{lemma}\label{lem:const}
Let $f:[a,b]\to \mathbb{R}$ be the constant function $f(x)=c$ for every $x\in[a,b]$. Then $$\int_a^b f(x)\, d \alpha= c\mu_\alpha([a,b]).$$
\end{lemma}

\begin{proof}
Note that $f$ is a step function with respect to the trivial partition $\{a, b\}$. Thus, \begin{align*}
    \int_a^b f(x)\, d \alpha &= f(a)\mu_\alpha(\{a\})+f(b)\mu_\alpha(b)+c\mu_\alpha((a, b))\\
    &=c(\alpha(a+)-\alpha(a-)+\alpha(b+)-\alpha(b-)+\alpha(b-)-\alpha(a+))\\
    &=c(\alpha(b+)-\alpha(a-))\\
    &=c\mu_\alpha([a,b]).
\end{align*}
\end{proof}

We now show that the usual properties of the Stieltjes integral hold: linearity, additivity, monotonicity, and linearity of integrator.  The proofs of these results are very similar to the proofs for the DS-integral, but with some differences.  To illustrate this we give the proofs of the first two results, and refer to~\cite{TSI} for the others.

\begin{theorem}\label{thm:steplin}
Given $f,g\in S[a,b]$, for any $r_1, r_2\in \mathbb{R}$, $$\int_{a}^{b}[r_1f(x)+r_2g(x)]\, d \alpha=r_1\int_{a}^{b}f(x)\, d \alpha+r_2\int_{a}^{b}g(x)\, d \alpha.$$
\end{theorem}

\begin{proof}
We argue as in \cite[Theorem 1.18, Theorem 4.7]{TSI}.
Given step functions $f, g$, by passing to a common refinement, we may assume both are defined with respect to the partition $\mathcal{R}=\{x_i\}_{i=0}^n$ where $f(x)=c_i,\; g(x)=d_i$ for $x\in I_i$. Then, 
\begin{align*}
    \int_{a}^{b}[r_1f(x)+r_2g(x)]\, d \alpha&= \sum_{i=0}^n [r_1f(x_i)+r_2g(x_i)]\mu_{\alpha}(\{x_i\})+\sum_{i=1}^n [r_1c_i+r_2d_i]\mu_{\alpha}(I_i)\\
    &=r_1\sum_{i=0}^n f(x_i)\mu_{\alpha}(\{x_i\})+r_2\sum_{i=0}^n g(x_i)\mu_{\alpha}(\{x_i\})\\&\quad +r_1\sum_{i=1}^n c_i\mu_{\alpha}(I_i)+r_2\sum_{i=1}^n d_i\mu_{\alpha}(I_i)\\
    &= r_1\int_{a}^{b}f(x)\, d \alpha+r_2\int_{a}^{b}g(x)\, d \alpha.
\end{align*}

\end{proof}

\begin{theorem}\label{thm:stepadd}
Given $f\in S[a,b]$, for any $m\in(a,b)$, $$\int_{a}^{m}f(x)\, d \alpha+\int_{m}^{b}f(x)\, d \alpha=\int_{a}^{b}f(x)\, d \alpha.$$
\end{theorem}

The corresponding result for the DS-integral is proved in~\cite[Theorem 1.18, Theorem 4.7]{TSI} but there are some differences in the proof for the RDS-integral.

\begin{proof}
By passing to a refinement, suppose $f$ is a step function on $[a,b]$ with respect to a partition $\mathcal{P}=\{x_i\}_{i=0}^n$ which contains $m$, so there exists $i_0$ where $x_{i_0}=m$. Then, $f$ is a step function on $[a,m]$ with respect to the partition $\{x_i\}_{i=0}^{i_0}$ and on $[m,b]$ with respect to the partition $\{x_i\}_{i=i_0}^{n}$. Then, \begin{align*}
    \int_{a}^b f(x)\, d \alpha &= \sum_{i=0}^n f(x_i)\mu_{\alpha}(\{x_i\})+\sum_{i=1}^nc_i\mu_{\alpha}(I_i) \\
    &= \sum_{i=0}^{i_0-1} f(x_i)\mu_{\alpha}(\{x_i\})+ f(m)[\alpha(m+)-\alpha(m-)]+ \sum_{i=i_0+1}^{n} f(x_i)\mu_{\alpha}(\{x_i\})\\&\quad + \sum_{i=1}^{i_0}c_i\mu_{\alpha}(I_i) +\sum_{i=i_0+1}^{n}c_i\mu_{\alpha}(I_i)\\
    &= \sum_{i=0}^{i_0-1} f(x_i)\mu_{\alpha}(\{x_i\}) + f(m)[\alpha(m)-\alpha(m-)]+ \sum_{i=1}^{i_0}c_i\mu_{\alpha}(I_i) \\&\quad + f(m)[\alpha(m+)-\alpha(m)]+\sum_{i=i_0+1}^{n} f(x_i)\mu_{\alpha}(\{x_i\})+\sum_{i=i_0+1}^{n}c_i\mu_{\alpha}(I_i).
\end{align*}
Note that on $[a,m]$, $\alpha(m+)= \alpha(m)$ and similarly on $[m,b]$, $\alpha(m-)= \alpha(m)$. Hence,
$$\int_{a}^b f(x)\, d \alpha=\int_{a}^m f(x)\, d \alpha+\int_{m}^b f(x)\, d \alpha.$$
\end{proof}

\begin{theorem}\label{thm:stepmon}
Given $\alpha\in I[a,b]$ and $f,g\in S[a,b]$, if for every $x\in[a,b]$, $f(x)\ge g(x)$, then $$\int_{a}^{b}f(x)\, d \alpha\ge \int_{a}^{b}g(x)\, d \alpha.$$
\end{theorem}

\begin{proof}
The proof is the same as for the DS-integral:  see\cite[Theorem 1.18, Theorem 4.7]{TSI}.
\end{proof}

\begin{theorem}\label{thm:steplinint}
Given $\alpha, \beta\in BV[a,b]$, let $r_1, r_2\in \mathbb{R}$ and define $\gamma = r_1\alpha+r_2\beta\in BV[a,b]$. Then, for every $f\in S[a,b]$, $$\int_{a}^{b}f(x)\, d \gamma=r_1\int_{a}^{b}f(x)\, d \alpha+r_2\int_{a}^{b}f(x)\, d \beta.$$ 
\end{theorem}

\begin{proof}
The proof is the same as for the DS-integral:  see\cite[Theorem 4.10]{TSI}.
\end{proof}

\section{The Ross-Darboux-Stieltjes Integral}
\label{section:rds-defn-bis}

In this section we  define the RDS-integral for an arbitrary bounded function $f$ and prove some basic properties of the integral.  Additional properties, proved using the so-called RDS criterion, Theorem~\ref{thm:RDScrit}, will be given in Section~\ref{section:RDS-criterion}.  We follow the approach in~\cite{TSI} as outlined in Section~\ref{section:prelim} and approximate the integral of $f$ by the integral of step functions that bracket it.  For the next definition recall  $S_U$ and $S_L$, given in Definition~\ref{def:bracket}.

\begin{definition}\label{def:RDSint}
Given $\alpha\in I[a,b]$ and $f\in B[a,b]$, define
\begin{align*}
    L_\alpha (f, [a,b])=\sup\left\{(S)\int_{a}^b v(x)\, d \alpha \; :\; v\in S_L(f, [a,b])\right\},\\
    U_\alpha (f, [a,b])=\inf\left\{(S)\int_{a}^b u(x)\, d \alpha \; :\; u\in S_U(f, [a,b])\right\}.
\end{align*}

If $L_\alpha(f, [a,b])=U_\alpha(f, [a,b])$, then we say $f$ is Ross-Darboux-Stieltjes integrable with respect to $\alpha$ and we denote their common value by
$$(RDS) \int_a^bf(x)\, d\alpha.$$ 
Moreover, for $\alpha\in BV[a,b]$, if there exists a decomposition $\alpha=\alpha^+ - \alpha^-$ with $\alpha^{\pm}\in I[a,b]$ where $f\in RDS_{\alpha^+}[a,b]$ and $f\in RDS_{\alpha^-}[a,b]$, then we define $f\in RDS_\alpha[a,b]$ and 
$$(RDS) \int_a^bf(x)\, d \alpha= (RDS) \int_a^bf(x)\, d \alpha^+ - (RDS) \int_a^bf(x)\, d \alpha^-.$$
We denote the set of functions which are Ross-Darboux-Stieltjes integrable with respect to $\alpha$ as $RDS_\alpha[a,b]$.
\end{definition}

For brevity, unless we need to work with different definitions of the integral, hereafter we will omit the notation "$(RDS)$" and simply write $\int_a^b f\,d\alpha$. 

\medskip

To work with this definition, we need to establish three things.  First, we have to show that it is well-defined in the sense that all of the objects in the definition exist.  

\begin{prop}\label{prop:lowerupperineq} 
Given $\alpha\in I[a,b]$ and $f\in B[a,b]$, $L_\alpha (f, [a,b])$ and $U_\alpha (f, [a,b])$ always exist and $L_\alpha (f, [a,b])\le U_\alpha (f, [a,b])$. 
\end{prop}

The same questions needed to be resolved for the DS-integral as defined above.  This is done in the discussion after~\cite[Definition 4.12]{TSI}.  We include the short proof for completeness.

\begin{proof}
Fix $f\in B[a,b]$. Then, there exists $m, M$ such that $m\le f(x) \le M$ for all $x\in [a,b].$ Hence, $u_M(x) = M $ is in $S_U(f,[a,b])$ and $v_m(x)=m$ is in $S_L(f, [a,b])$, so these sets are non-empty. Moreover, fix $u\in S_U(f,[a,b])$ and $v\in S_L(f,[a,b])$. Since $v(x) \le f(x) \le u(x)$ for all $x\in [a,b]$, we have by Theorem \ref{thm:stepmon}, 
$$(S)\int_a^b v(x)\, d \alpha \le (S)\int_a^b u(x)\, d \alpha.$$ 
Thus, the set of integrals of step functions in $S_U(f, [a,b])$ is bounded below by $(S)\int_a^b v(x)\, d \alpha$ so the infimum $U_\alpha (f, [a,b])$ exists and $$(S)\int_a^b v(x)\, d \alpha \le U_\alpha (f, [a,b]).$$ From here, we see the set of $\alpha$-integrals of step functions in $S_L(f, [a,b])$ is bounded above by $U_\alpha (f, [a,b])$ so the supremum $L_\alpha (f, [a,b])$ exists as well and $$L_\alpha (f, [a,b]) \le U_\alpha (f, [a,b]).$$ 
\end{proof}

Second, we must show that given a step function, the integrals defined by~Definitions~\ref{def:stepint} and~\ref{def:RDSint} agree.  Again a similar question is addressed after~\cite[Definition 4.12]{TSI} for the DS-integral.

\begin{prop}\label{prop:stepagree}
Given $\alpha\in BV[a,b]$ and $f\in S[a,b]$, $f\in RDS_\alpha[a,b]$ and 
$$(S)\int_a^b f(x)\, d \alpha = \int_a^b f(x)\, d \alpha.$$
\end{prop}

\begin{proof}
 First, suppose $\alpha \in I[a,b]$. Fix $f\in S[a,b]$. Then we have that  $f\in S_L(f, [a,b])$ and $f\in S_U(f, [a,b])$. Hence, by Proposition~\ref{prop:lowerupperineq},
 \begin{equation*}
    (S)\int_a^b f(x)\, d \alpha \le L_\alpha (f, [a,b])\le U_\alpha (f, [a,b]) \le (S)\int_a^b f(x)\, d \alpha.
\end{equation*} 
Therefore, $$L_\alpha (f, [a,b]) = U_\alpha (f, [a,b])=(S)\int_a^b f(x)\, d \alpha$$ as desired.

Now, suppose $\alpha \in BV[a,b]$. Then there exists a decomposition $\alpha=\alpha^+ - \alpha^-$. By the previous case, $f\in RDS_{\alpha^+}[a,b]$ and $f\in RDS_{\alpha^-}[a,b]$. Hence, $f\in RDS_\alpha[a,b]$ and by Theorem~ \ref{thm:steplinint}, 
\begin{multline*}
    \int_a^b f(x)\, d \alpha
    =\int_a^b f(x)\, d \alpha^+-\int_a^b f(x)\, d \alpha^-\\
    =(S)\int_a^b f(x)\, d \alpha^+-(S)\int_a^b f(x)\, d \alpha^-
    =(S)\int_a^b f(x)\, d \alpha.
\end{multline*}
\end{proof}

\begin{remark}
Given Proposition~\ref{prop:stepagree}, hereafter we will use either definition when computing the integral of a step function.
\end{remark}

Third, we need to show that given $\alpha\in BV[a,b]$ the definition of the RDS-integral is independent of the Jordan decomposition of $\alpha$ used.  This requires more work; the final proof is in Theorem~\ref{prop:decompinv} below.  To prove this result we need some preliminary results.  The first is the following characterization of RDS-integrability.  This corresponds to the Darboux criterion for the Darboux integral and the corresponding result for the DS-integral (see~\cite[Theorems~1.27, 4.17]{TSI}) and we will refer to it as  the Ross-Darboux-Stieltjes criterion or RDS criterion.  The short proof is nearly identical to the proofs of these results; we include the details for completeness. 

\begin{theorem}\label{thm:RDScrit}
Given $\alpha\in I[a,b]$ and $f\in B[a,b]$, $f\in RDS_\alpha[a,b]$ if and only if for every $\epsilon>0$ there exists $u,\,v\in S[a,b]$ which bracket $f$, that is,  $v(x)\le f(x) \le u(x)$ for each $x\in [a,b]$, and $$\int_a^b u(x)-v(x)\, d \alpha < \epsilon.$$
\end{theorem}

\begin{proof}
Suppose $f\in RDS_\alpha [a,b]$. Fix $\epsilon>0$. By Definition \ref{def:RDSint} and the properties of infimum and supremum, there exists $u,\,v\in S[a,b]$ such that $v(x)\le f(x) \le u(x)$ and $$\int_a^b u(x)\, d \alpha -\frac{\epsilon}{2} < \int_a^b f(x)\, d \alpha < \int_a^b v(x)\, d \alpha +\frac{\epsilon}{2}.$$ We rearrange and use Theorem \ref{thm:steplin} to get the desired inequality.

Now for the converse, suppose for each $\epsilon>0$, there exists $u,\,v\in S[a,b]$ such that $v(x)\le f(x) \le u(x)$ for each $x\in [a,b]$ and $$\int_a^b u(x)-v(x)\, d \alpha < \epsilon.$$ Fix such $\epsilon>0$. Then, $$\int_a^b v(x)\, d \alpha \le L_\alpha(f, [a,b]) \le U_\alpha(f, [a,b]) \le \int_a^b u(x)\, d \alpha.$$ Hence, $$U_\alpha(f, [a,b])-L_\alpha(f, [a,b])\le \int_a^b u(x)\, d \alpha -\int_a^b v(x)\, d \alpha < \epsilon.$$ Since $\epsilon$ is arbitrary, this is true for all $\epsilon>0$ and thus, $U_\alpha(f, [a,b])=L_\alpha(f, [a,b])$. By Definition \ref{def:RDSint}, we have the result.
\end{proof}

Before continuing our proof of Theorem~\ref{prop:decompinv}, we pause to give two corollaries of the RDS criterion.  The first will be used below in later sections.  
 The proof is identical to that for the DS-integral (see \cite[Corollaries 1.28, 4.18]{TSI}) and so we omit the details.

\begin{corollary}\label{cor:RDScritbdd}
    Given $\alpha\in I[a,b]$, if $f\in RDS_\alpha [a,b]$ and $|f(x)|\le M$ for every $x\in[a,b]$, then for any $\epsilon>0$, there exists step functions $u,\,v$ such that $$-M\le v(x)\le f(x)\le u(x)\le M$$ and $$\int_a^b u(x)-v(x)\, d \alpha<\epsilon.$$
\end{corollary}


As our second corollary we prove that our definition of the RDS-integral in~Definition~\ref{def:RDSint} in terms of step functions is equivalent to the definition by Ross in Definition~\ref{def: RossIntegral} in terms of upper and lower sums.   To do so, we introduce the notion of  best-fit step functions (see~\cite[Definition 1.30]{TSI}).

\begin{definition}\label{def:bestfitstepfunctions}
    Given $f\in B[a,b]$ and a partition $\mathcal{P}$ of $[a,b]$, define the best-fit step functions of $f$ with respect to $\mathcal{P}$, denoted $\hat{u}, \hat{v}\in S[a,b]$, as follows: on each partition interval $I_i$, let $$M_i=\sup\{f(x)\; :\; x\in I_i\},\quad m_i=\inf\{f(x)\; :\; x\in I_i\},$$ and define $\hat{u}(x)=M_i$ and $\hat{v}(x)= m_i$ for $x\in I_i$. At each partition point $x_i\in \mathcal{P}$, let $\hat{u}(x_i)=\hat{v}(x_i)= f(x_i)$.
\end{definition}

Observe that for any $\alpha\in I[a,b]$, 
$$\int_a^b \hat{u}(x)\, d\alpha = U_\alpha^r(f,\mathcal{P}) \:\text{ and } \; \int_a^b \hat{v}(x)\, d\alpha = L_\alpha^r(f,\mathcal{P}).$$ 
Moreover, if $u,\,v$ are step functions with respect to the partition $\mathcal{P}$ that bracket $f$, then we immediately have that for all $x\in [a,b]$, 
\begin{equation} \label{eqn:bestfit-bracket}
v(x) \leq \hat{v}(x) \leq f(x) \leq \hat{u}(x) \leq u(x).
\end{equation}

Using this observation, we can pass between the RDS criterion in Theorem~\ref{thm:RDSLebcrit} and the corresponding result proved by Ross~\cite[Theorem 35.6]{Ross} using his definition.  As an immediate consequence we have that the two definitions are equivalent.  To distinguish them, temporarily denote the set of functions integrable by Definition~\ref{def: RossIntegral} by $RDS^*_\alpha[a,b]$.

\begin{corollary}\label{cor:bestfitRDScrit}
    Given $\alpha\in I[a,b]$, $f\in RDS_\alpha[a,b]$ if and only if for any $\epsilon>0$, there exists a partition $\mathcal{P}$ and best-fit step function $\hat{u},\hat{v}$ of $f$ with respect to $\mathcal{P}$ such that 
    \begin{equation} \label{eqn:bestfitRDScrit1}
U_\alpha^r(f,\mathcal{P}) - L_\alpha^r(f,\mathcal{P}) = \int_a^b \hat{u}(x)-\hat{v}(x)\, d\alpha<\epsilon.
\end{equation}
    As a consequence, $f\in RDS_\alpha[a,b]$ if and only if $f\in RDS_\alpha^*[a,b]$ and the two integrals agree.
\end{corollary}

\begin{proof}
    We argue as for the Darboux integral in \cite[Corollary 1.32]{TSI}. Suppose first that $f\in RDS_\alpha[a,b]$. Fix $\epsilon>0$. By the RDS criterion, Theorem \ref{thm:RDScrit}, there exist $u,\,v\in S[a,b]$ which bracket $f$ such that 
    $$\int_a^b u(x)-v(x)\; d\alpha<\epsilon.$$ 
    We may assume that $u$ and $v$ are step functions with respect to a common partition $\mathcal{P}$. Let $\hat{u}$ and $\hat{v}$ be the best-fit step functions of $f$ with respect to $\mathcal{P}$. Then by \eqref{eqn:bestfit-bracket}, we have that $\hat{u}(x)-\hat{v}(x)\leq u(x)-v(x)$, and so 
     by Theorem \ref{thm:stepmon}, 
     $$ \int_a^b \hat{u}(x)-\hat{v}(x)\, d\alpha \le \int_a^b u(x)-v(x)\, d\alpha<\epsilon.$$
    Therefore, by \cite[Theorem 35.6]{Ross}, $f\in RDS^*_\alpha[a,b]$.

    Conversely, if~\eqref{eqn:bestfitRDScrit1} holds for every $\epsilon>0$, then, since $\hat{u}$ and $\hat{v}$ are step functions that bracket $f$, by Theorem~\ref{thm:RDScrit} we have $f\in RDS_\alpha[a,b]$.
    
    Finally, to see that the integrals agree, note that given $\epsilon>0$ and the best fit step functions such that~\eqref{eqn:bestfitRDScrit1} holds, then by their respective definitions we have that
    \[ \bigg| (RDS) \int_a^b f(x)\, d\alpha - (RDS^*) \int_a^b f(x)\, d\alpha \bigg| 
    \leq \int_a^b \hat{u}(x)-\hat{v}(x)\, d\alpha < \epsilon. \]
    Since $\epsilon>0$ is arbitrary, the two integrals are equal.
\end{proof}

\begin{lemma}\label{lem:constint}
Given $\alpha\in BV[a,b]$ such that $\alpha(x)=c$ is constant. Then, for every $f\in B[a,b]$, $f\in RDS_\alpha[a,b]$ and 
\begin{equation} \label{eqn:zero-int}
\int_{a}^b f(x)\, d \alpha =0.
\end{equation}
\end{lemma}

\begin{proof}
We argue as for the DS-integral in~\cite[Theorem 4.20]{TSI}.  Note $\alpha\in I[a,b]$. Fix $f\in B[a,b]$. Fix $u\in S_U(f, [a,b])$ defined with respect to a partition $\mathcal{P}=\{x_i\}_{i=0}^n$ of $[a,b]$ For each $i$, 
\[
\mu_\alpha(\{x_i\})=\alpha(x_i +)-\alpha(x_i -)= c - c=0;
\]
similarly, $\mu_\alpha(I_i)=0$. Hence, for every $u\in S_U(f, [a,b])$, $$\int_a^b u(x)\, d \alpha = 0$$ and so $U_\alpha(f, [a,b])=0$. By the same argument, $L_\alpha(f, [a,b])=0$. Thus, $f\in RDS_\alpha[a,b]$ and~\eqref{eqn:zero-int} holds.
\end{proof}

The proof of the following lemma is identical to the proof for the DS-integral~\cite[Prop 4.21]{TSI}, and we omit the details.  

\begin{lemma}\label{lem:linintinc}
Given $\alpha, \beta\in I[a,b]$ and $r_1, r_2
\ge 0$, define $\gamma = r_1\alpha+r_2\beta\in I[a,b]$. If $f\in RDS_\alpha[a,b]$ and $f\in RDS_\beta[a,b]$, then $f\in RDS_\gamma[a,b]$ and $$\int_{a}^{b}f(x)\, d \gamma=r_1\int_{a}^{b}f(x)\, d \alpha+r_2\int_{a}^{b}f(x)\, d \beta.$$ 
\end{lemma}

\begin{corollary}
    Given $\alpha\in I[a,b]$ and $c\in \mathbb{R}$, if $f\in RDS_{\alpha}[a,b]$, then $f\in RDS_{\alpha+c}[a,b]$ and $$\int_a^b f(x)\, d \alpha = \int_a^b f(x)\, d [\alpha+c].$$
\end{corollary}

\begin{proof}
    This follows immediately from Lemma \ref{lem:constint} and Lemma \ref{lem:linintinc}.
\end{proof}

We can now prove the independence of the RDS-integral from the decomposition of the integrator $\alpha$ as the difference of increasing functions.   The proof is essentially the same as for the DS-integral~\cite[Lemma 4.27]{TSI}, but we include the short proof for completeness.

\begin{theorem}\label{prop:decompinv}
For $\alpha\in BV[a,b]$, given two different decompositions $\alpha=\alpha^+-\alpha^-=\beta^+ - \beta^-,$ if $f\in RDS_{\alpha}[a,b]$ with respect to both decompositions, then 
$$\int_a^bf(x)\, d \alpha^+-\int_a^bf(x)\, d \alpha^-=\int_a^bf(x)\, d \beta^+ - \int_a^bf(x)\, d \beta^-.$$
\end{theorem}

\begin{proof}
Let $\gamma=\alpha^+ + \beta^- = \alpha^- + \beta^+$. Then $\gamma\in I[a,b]$ and by Lemma \ref{lem:linintinc}, 
\[
    \int_a^bf(x)\, d \alpha^+ +\int_a^bf(x)\, d \beta^- =\int_a^bf(x)\, d \gamma =\int_a^bf(x)\, d \alpha^- +\int_a^bf(x)\, d \beta^+.
\]

If we rearrange terms, we get the desired equality.
\end{proof}

\medskip

By Theorem~\ref{prop:decompinv}, when computing the  RDS-integral we can choose any decomposition of  $\alpha$ for which it is defined.  Determining such a decomposition may seem like a problem, but we will show that if a function is RDS-integrable with respect to any decomposition of $\alpha$, then it is always integrable with respect to a canonical decomposition of $\alpha$ in terms of its positive and negative variation.  The proof requires a definition and two lemmas.  For the first definition, recall the definition of the variation $V(f,[a,b])$ given in Definition~\ref{defn:BV}.  

\begin{definition}\label{def:pnalpha}
    Fix $\alpha\in BV[a,b]$, and define the variation function of $\alpha$, $V\alpha$, by 
    \[ V\alpha(x) = V(\alpha, [a,x]). \]
    Further, define the positive and negative variations of $\alpha$ by 
    $$P\alpha=\frac{1}{2}(V\alpha +\alpha-\alpha(a)) \qquad \text{and} \qquad N\alpha=\frac{1}{2}(V\alpha -\alpha+\alpha(a)).$$
\end{definition}

The functions $V\alpha,\,P\alpha,\,N\alpha$ are all increasing and yield a Jordan decomposition of $f$.  For a proof, see~\cite[Section~3.3]{TSI}.

\begin{lemma} \label{lemma:var-functions}
Given $\alpha \in BV[a,b]$, $V\alpha,\,P\alpha,\,N\alpha\in I[a,b]$ and 
$$\alpha(x)=P\alpha(x)-N\alpha(x) +\alpha(a).$$
\end{lemma}

The second lemma is a kind of converse to Lemma~\ref{lem:linintinc}.

\begin{lemma}\label{prop:linintconv}
Given $\alpha, \beta\in I[a,b]$ and $r_1, r_2
> 0$, define $\gamma = r_1\alpha+r_2\beta\in I[a,b]$. If $f\in RDS_\gamma[a,b]$, then $f\in RDS_\alpha[a,b]$ and $f\in RDS_\beta[a,b].$
\end{lemma}

\begin{proof}
It suffices to show $f\in RDS_\alpha[a,b]$ since by a symmetric argument, $f\in RDS_\beta[a,b]$. Fix $\epsilon>0$. By Theorem \ref{thm:RDScrit}, there exist step functions $u,\,v$ which bracket $f$ such that $$\int_a^b u(x)-v(x)\, d \gamma <r_1\epsilon.$$ Then, since $u-v$ is nonnegative and $r_2\beta$ is increasing, by Theorems~\ref{thm:stepmon} and~\ref{thm:steplinint}, 
\begin{align*}
    \int_a^b u(x)-v(x)\, d \alpha &= \frac{1}{r_1}\int_a^b u(x)-v(x)\, d (r_1\alpha)\\
    &\le \frac{1}{r_1}\left(\int_a^b u(x)-v(x)\, d (r_1\alpha)+\int_a^b u(x)-v(x)\, d (r_2\beta)\right)\\
    &=\frac{1}{r_1}\int_a^b u(x)-v(x)\, d \gamma\\
    &<\epsilon.
\end{align*}
Hence, by Theorem \ref{thm:RDScrit}, $f\in RDS_\alpha[a,b].$
\end{proof}

\begin{theorem}\label{thm:pnalpha}
Given $\alpha\in BV[a,b]$, $f\in RDS_\alpha[a,b]$ if and only if $f\in RDS_{P\alpha}[a,b]$, $f\in RDS_{N\alpha}[a,b]$. Moreover, $$\int_a^b f(x)\, d \alpha = \int_a^b f(x)\, d P\alpha-\int_a^b f(x)\, d N\alpha.$$
\end{theorem}

\begin{proof}
We argue as in~\cite[Theorem 4.28]{TSI} for the DS-integral.  Suppose first that $f\in RDS_\alpha[a,b]$. By Lemma~\ref{lemma:var-functions}, $P\alpha, N\alpha \in I[a,b]$. Since $f\in RDS_\alpha[a,b]$, there exists a decomposition $\alpha=\alpha^+-\alpha^-$ such that $f\in RDS_{\alpha^+}[a,b]$ and $f\in RDS_{\alpha^-}[a,b]$. Thus, for all $x\in [a,b]$ 
$$\alpha(x)= [\alpha^+(x)-\alpha^+(a)]-[\alpha^-(x) -\alpha^-(a)]+\alpha(a).$$  
By \cite[Theorem 3.21]{TSI}, there exists $\beta\in I[a,b]$ such that $\alpha^+(x) - \alpha^+(a) = P\alpha(x) + \beta(x)$ and $\alpha^-(x) - \alpha^-(a) = N\alpha(x)+\beta(x)$. By Lemma \ref{lem:linintinc}, $f\in RDS_{\alpha^+-\alpha^+(a)}[a,b]= RDS_{P\alpha+\beta}[a,b]$. Thus, by Lemma~\ref{prop:linintconv}, $f\in RDS_{P\alpha}[a,b]$, $f\in RDS_{N\alpha}[a,b]$, and $f\in RDS_\beta[a,b]$. 

Now suppose $f\in RDS_{P\alpha}[a,b]$ and $f\in RDS_{N\alpha}[a,b]$.  Since $P\alpha +\alpha(a)\in I[a,b]$. By Lemmas~\ref{lem:constint} and~\ref{lem:linintinc}, 
$f\in RDS_{P\alpha +\alpha(a)}[a,b]$. Then, by Definition \ref{def:RDSint}, since $\alpha = P\alpha -N\alpha +\alpha(a)$, $f\in RDS_{\alpha}[a,b].$
Moreover,
\begin{align*}
    \int_a^b f(x)\, d \alpha^+ &=\int_a^b f(x)\, d [\alpha^+(a)]+\int_a^b f(x)\, d P\alpha+\int_a^b f(x)\, d \beta\\ &=\int_a^b f(x)\, d P\alpha+\int_a^b f(x)\, d \beta,
\end{align*}
 and \begin{align*}
    \int_a^b f(x)\, d \alpha^- &=\int_a^b f(x)\, d [\alpha^-(a)]+\int_a^b f(x)\, d N\alpha+\int_a^b f(x)\, d \beta\\ &=\int_a^b f(x)\, d N\alpha+\int_a^b f(x)\, d \beta.
\end{align*}
Hence, 
$$\int_a^b f(x)\, d \alpha =\int_a^b f(x)\, d \alpha^+-\int_a^b f(x)\, d \alpha^- = \int_a^b f(x)\, d P\alpha-\int_a^b f(x)\, d N\alpha.$$
\end{proof}

\medskip

We conclude this section by showing that the DS-integral given in~Definition \ref{def:DSint} and the RDS-integral agree if the integrator or integrand are continuous.  We will give a complete characterization of the relationship between these two  definitions of the integral in Section~\ref{section:alt-defns}.  
To state our first result, let $CBV[a,b]$ be the set of functions which are both continuous and of bounded variation on $[a,b]$. 

\begin{theorem}\label{thm:RDSiffDSintegrator}
    Given $\alpha\in CBV[a,b]$, $f\in RDS_{\alpha}[a,b]$ if and only if $f\in DS_\alpha [a,b]$ and in this case, 
    \begin{equation} \label{eqn:rdsiffdsintegrator1}
\int_a^b f(x)\, d \alpha = (DS)\int_a^b f(x)\, d \alpha.
\end{equation}
\end{theorem}

\begin{proof}
First, consider $\alpha \in I[a,b]$. Suppose $f\in RDS_\alpha [a,b]$. Fix $u\in S[a,b]$. Since $\alpha\in C[a,b]$, $$\int_a^b u(x)\, d \alpha = (DS)\int_a^b u(x)\, d \alpha.$$ By Definition \ref{def:DSint} and Definition \ref{def:RDSint}, we have $$L^i_\alpha(f,[a,b])= L_\alpha(f,[a,b])= U_\alpha(f,[a,b])= U^i_\alpha(f,[a,b]).$$  Thus, $f\in DS[a,b]$. The reverse implication is proved in the same way, and it is immediate that \eqref{eqn:rdsiffdsintegrator1} holds.

Now, consider $\alpha \in BV[a,b]$. Suppose $f\in RDS_\alpha [a,b]$. Then, by Theorem \ref{thm:pnalpha}, $f\in RDS_{P\alpha}[a,b]$ and $f\in RDS_{N\alpha}[a,b]$. By \cite[Definition 3.16]{TSI}, it is easy to see that $P\alpha, N\alpha \in CBV[a,b]$. By the previous case, $f\in DS_{P\alpha}[a,b]$ and $f\in DS_{N\alpha}[a,b]$. By \cite[Theorem 4.28]{TSI}, $f\in DS_\alpha [a,b]$. The converse follows similarly since all these implications are reversible. By the previous case, $$\int_a^b f(x)\, d P\alpha = (DS)\int_a^b f(x)\, d P\alpha,$$ $$\int_a^b f(x)\, d N\alpha = (DS)\int_a^b f(x)\, d N\alpha.$$ Therefore, by Theorem \ref{thm:pnalpha}
 and \cite[Theorem 4.28]{TSI}, \eqref{eqn:rdsiffdsintegrator1} holds.
\end{proof}

\begin{theorem}\label{thm:RDSiffDSintegrand}
    Suppose $f\in C[a,b]$ and $ \alpha\in BV[a,b]$. Then, $f\in RDS_{\alpha}[a,b]$ if and only if $f\in DS_\alpha [a,b]$ and in this case, \eqref{eqn:rdsiffdsintegrator1} holds.
\end{theorem}

\begin{proof}
    By Theorem \ref{thm:continuousinRDS} below and \cite[Theorem 4.35]{TSI}, if $f\in C[a,b]$, then $f\in RDS_\alpha[a,b]$ and $f\in DS_\alpha[a,b]$. It suffices to show the integrals are equal. 

    First, consider $\alpha\in I[a,b]$ . If $\mu_\alpha([a,b])= \alpha(b)-\alpha(a)=0$, then $\alpha$ is constant. Hence, by Lemma \ref{lem:constint} and \cite[Lemma 4.20]{TSI}, we have both integrals are $0$. Otherwise, suppose $\mu_\alpha([a,b])=\alpha(b)-\alpha(a)>0$. Fix $\epsilon>0$. Since $f$ is continuous on $[a,b]$, it is uniformly continuous. Thus, there exists $\delta>0$ such that for every $x,y\in [a,b]$ with $|x-y|<\delta$, $$|f(x)-f(y)|<\frac{\epsilon}{\mu_\alpha([a,b])}.$$ Fix $N\in \mathbb{N}$ such that $N>\frac{b-a}{\delta}.$ Moreover, let $\mathcal{P}_N=\{x_i\}_{i=1}^N=\{a+\frac{b-a}{N}i\}_{i=0}^N$ be a regular partition of $[a,b]$. Hence, $\overline{I_i}=[x_{i-1}, x_i]$ is compact so $f$ attains its extrema on each interval. For each $1\le i\le N$, let 
    $$M_i = \sup\{f(x)\; :\; x\in \overline{I_i}\}=f(x_i^M)$$ 
    and 
    $$m_i = \inf\{f(x)\; :\; x\in \overline{I_i}\}=f(x_i^m).$$ 
    Define step functions $u,\,v$ by 
    \begin{equation} \label{eqn:RDSiffDSintegrand1}
u(x)=\begin{cases}M_i,& x\in I_i,\\ f(x_i), & x= x_i,\end{cases} \quad  \text{and} 
\quad v(x)=\begin{cases}m_i,& x\in I_i,\\ f(x_i), & x= x_i.\end{cases}
\end{equation}
%
Observe that 
\begin{align*}
      &  \int_a^b u(x)\, d \alpha - (DS)\int_a^b v(x)\, d \alpha \\
      & \qquad \qquad = \sum_{i=0}^N f(x_i)\mu_\alpha(\{x_i\}) + \sum_{i=1}^N M_i\mu_\alpha(I_i) - \sum_{i=1}^Nm_i\alpha(I_i)\\
        & \qquad \qquad= \sum_{i=0}^{N-1} f(x_i)[\alpha(x_i+)-\alpha(x_i)]+\sum_{i=1}^N f(x_i)[\alpha(x_i)-\alpha(x_i-)]\\ & \qquad \qquad\quad + \sum_{i=1}^N M_i\mu_\alpha(I_i)- \sum_{i=1}^Nm_i\alpha(I_i)\\
        & \qquad \qquad\le \sum_{i=0}^{N-1} M_{i+1}[\alpha(x_i+)-\alpha(x_i)]+\sum_{i=1}^N M_{i}[\alpha(x_i)-\alpha(x_i-)]\\ & \qquad \qquad\quad + \sum_{i=1}^N M_i[\alpha(I_i) + \alpha(x_i-)-\alpha(x_i) -\alpha(x_{i-1}+)+\alpha(x_{i-1})]\\ & \qquad \qquad\quad- \sum_{i=1}^Nm_i\alpha(I_i)\\
        & \qquad \qquad=\sum_{i=1}^N M_i\alpha(I_i)-\sum_{i=1}^N m_i\alpha(I_i)\\
        & \qquad \qquad< \frac{\epsilon}{\mu_\alpha([a,b])}\sum_{i=1}^N \alpha(I_i)\\
        & \qquad \qquad=\frac{\epsilon}{\mu_\alpha([a,b])}(\alpha(b)-\alpha(a))\\
        & \qquad \qquad=\epsilon.
    \end{align*}

    \noindent A similar argument yields $$(DS)\int_a^b u(x)\, d \alpha - \int_a^b v(x)\, d \alpha<\epsilon.$$ Hence, 
    $$\left|(DS)\int_a^b f(x)\, d \alpha-\int_a^b f(x)\, d \alpha\right|<\epsilon.$$ Since $\epsilon>0$ is arbitrary, \eqref{eqn:rdsiffdsintegrator1} holds.

    Now, we consider the case $\alpha\in BV[a,b]$. By the previous case, $$\int_a^b f(x)\, d P\alpha = (DS)\int_a^b f(x)\, d P\alpha,$$ $$\int_a^b f(x)\, d N\alpha = (DS)\int_a^b f(x)\, d N\alpha.$$ Therefore, by Theorem \ref{thm:pnalpha}
 and \cite[Theorem 4.28]{TSI}, \eqref{eqn:rdsiffdsintegrator1} holds. 
\end{proof}

\section{Consequences of the RDS Criterion}
\label{section:RDS-criterion}

In this section we develop the basic properties of the RDS-integral as a consequence of the RDS Criterion, Theorem~\ref{thm:RDScrit}.  In particular, we show that the RDS-integral is linear, additive, monotonic, and linear in the integrator.  In several cases the proofs are essentially identical to the proofs of the corresponding results for the Darboux integral and DS-integral given in \cite{TSI}, and so we omit the details.  We will instead concentrate on those results where some additional care is needed because of the additional term in the definition of the RDS-integral of step functions.

\begin{theorem}\label{thm:linearity}
(Linearity) Given $\alpha\in BV[a,b]$ and $f,g\in RDS_\alpha[a,b]$, for any $r_1, r_2\in \mathbb{R}$, $r_1f+r_2g\in RDS_\alpha[a,b]$ and $$\int_{a}^{b}[r_1f(x)+r_2g(x)]\, d \alpha=r_1\int_{a}^{b}f(x)\, d \alpha+r_2\int_{a}^{b}g(x)\, d \alpha.$$
\end{theorem}

The proof of Theorem \ref{thm:linearity} is an immediate consequence of Lemmas~\ref{lem:linaddition} and~\ref{lem:linscalar}.  The proof of the first is identical to the proof for the DS-integral in~\cite[Theorem 4.19, Theorem 4.30]{TSI}; the proof of the second is identical to the proofs in~\cite[Theorem 1.39, Theorem 4.28, Theorem 4.30]{TSI}.

 \begin{lemma}\label{lem:linaddition}
 Given $\alpha\in BV[a,b]$ and $f,g\in RDS_\alpha[a,b]$, $f+g\in RDS_\alpha[a,b]$ and $$\int_{a}^{b}f(x)+g(x)\, d \alpha=\int_{a}^{b}f(x)\, d \alpha+\int_{a}^{b}g(x)\, d \alpha.$$
 \end{lemma}

\begin{lemma}\label{lem:linscalar}
Given $\alpha\in BV[a,b]$ and $f\in RDS_\alpha[a,b]$, for any $c\in\mathbb{R}$, $cf\in RDS_\alpha[a,b]$ and $$\int_{a}^{b}cf(x)\, d \alpha=c\int_{a}^{b}f(x)\, d \alpha.$$
\end{lemma}

\begin{theorem}\label{thm:additivity}
(Additivity) Given $\alpha\in BV[a,b]$, for any $m\in(a,b)$, $f\in RDS_\alpha[a,b]$ if and only if $f\in RDS_\alpha[a,m]$ and $f\in RDS_\alpha[m,b]$. Furthermore,  
\begin{equation} \label{eqn:additivity1}
\int_{a}^{m}f(x)\, d \alpha+\int_{m}^{b}f(x)\, d \alpha=\int_{a}^{b}f(x)\, d \alpha.
\end{equation}
\end{theorem}

\begin{proof}
The proof of the first part is similar to that for the DS-integral in~\cite[Theorem 1.39, Theorem 4.19, Theorem 4.30]{TSI}, but the second part requires more care. First, consider $\alpha\in I[a,b]$. Fix $m\in (a,b)$ and suppose $f\in RDS_\alpha[a,b]$. Fix $\epsilon>0$. Then by Theorem \ref{thm:RDScrit}, there exists $u,\,v\in S[a,b]$ which bracket $f$ such that $$\int_{a}^b u(x)-v(x)\, d \alpha <\epsilon.$$ Thus, $u|_{[a,m]}, v|_{[a, m]}\in S[a,b]$ and they bracket $f$ on $[a,m]$. Moreover, by Theorem \ref{thm:stepmon}, we have $$\int_{m}^b u(x)-v(x)\, d \alpha \ge 0,$$ and by Theorem \ref{thm:stepadd},
\begin{multline*}
    \int_{a}^m u(x)-v(x)\, d \alpha \\
    \le \int_{a}^m u(x)-v(x)\, d \alpha+\int_{m}^b u(x)-v(x)\, d \alpha
    =\int_{a}^b u(x)-v(x)\, d \alpha
    <\epsilon.
\end{multline*}
Hence, by Theorem \ref{thm:RDScrit}, $f\in RDS_\alpha[a,m]$. By a similar arguement, $f\in RDS_\alpha [m, b]$.

Now, fix $f$ such that $f\in RDS_\alpha[a,m]$ and $f\in RDS_\alpha[m,b]$. Fix $\epsilon>0$. By Theorem~\ref{thm:RDScrit}, there exists $u_1,v_1\in S[a,m]$ and $u_2, v_2\in [m, b]$ such that both pairs bracket $f$ and $$\int_{a}^m u_1(x)-v_1(x)\, d \alpha<\frac{\epsilon}{2},\quad \int_{m}^b u_2(x)-v_2(x)\, d \alpha<\frac{\epsilon}{2} .$$ We combine $u_1$ and $u_2$ as follows: let
$$u(x)=\begin{cases} u_1(x), & x\in[a,m),\\ f(m), & x=m,\\ u_2(x), & x\in (m,b]. \end{cases} $$ Similarly, we combine $v_1$ and $v_2$ as follows: let $$v(x)=\begin{cases} v_1(x), & x\in[a,m),\\ f(m), & x=m,\\ v_2(x), & x\in (m,b]. \end{cases} $$

It is easy to see $v(x) \le f(x) \le u(x)$ for all $x\in [a,b]$. Also, observe that on $[a,m]$, $u(x)-v(x)\le u_1(x)-v_1(x)$ and on $[m,b]$,  $u(x)-v(x)\le u_2(x)-v_2(x).$ Thus, by Theorems~\ref{thm:stepadd} and~\ref{thm:stepmon}, 
\begin{multline*}
    \int_{a}^b u(x)-v(x)\, d \alpha 
     = \int_{a}^m u(x)-v(x)\, d \alpha+\int_{m}^b u(x)-v(x)\, d \alpha\\
    \le \int_{a}^m u_1(x)-v_1(x)\, d \alpha+\int_{m}^b u_2(x)-v_2(x)\, d \alpha
    < \frac{\epsilon}{2}+\frac{\epsilon}{2}
    =\epsilon.
\end{multline*} 
Hence, by Theorem \ref{thm:RDScrit}, $f\in RDS_\alpha[a,b]$.
Further,  by Definition \ref{def:RDSint}, 
$$\int_{a}^b v(x)\, d \alpha\le \int_{a}^b f(x)\, d \alpha \le \int_{a}^b u(x)\, d \alpha$$ and $$\int_{a}^m v(x)\, d \alpha+\int_{m}^b v(x)\, d \alpha\le \int_{a}^m f(x)\, d \alpha+\int_{m}^b f(x)\, d \alpha\le \int_{a}^m u(x)\, d \alpha+\int_{m}^b u(x)\, d \alpha.$$
Therefore,
$$\left|\int_{a}^m f(x)\, d \alpha+\int_{m}^b f(x)\, d \alpha-\int_{a}^b f(x)\, d \alpha\right|\le\epsilon$$ for all $\epsilon>0$, and so \eqref{eqn:additivity1} holds.

\medskip

Now, consider $\alpha\in BV[a,b]$. Fix $m\in (a,b)$. Suppose $f\in RDS_\alpha[a,b]$. Then by Theorem \ref{thm:pnalpha}, $f\in RDS_{P\alpha}[a,b]$ and $f\in RDS_{N\alpha}[a,b]$. By the previous case, $f\in RDS_{P\alpha}[a,m], f\in RDS_{P\alpha}[m,b], f\in RDS_{N\alpha}[a,m]$, and $f\in RDS_{N\alpha}[m,b].$ Then by Theorem \ref{thm:pnalpha}, $f\in RDS_{\alpha}[a,m]$ and $f\in RDS_{\alpha}[m,b]$. The reverse implications hold as well. Moreover, \begin{align*}
    \int_{a}^m f(x)\, d \alpha+\int_{m}^b f(x)\, d \alpha&=\int_{a}^m f(x)\, d P\alpha +\int_{m}^b f(x)\, d P\alpha\\ &\quad\quad -\int_{a}^m f(x)\, d N\alpha -\int_{m}^b f(x)\, d N\alpha\\
    &=\int_{a}^b f(x)\, d P\alpha -\int_{a}^b f(x)\, d N\alpha\\
    &=\int_{a}^b f(x)\, d \alpha.
\end{align*} 

\end{proof}

The proof of the following result is identical to the proof for the DS-integral in~\cite[Theorem 1.39, Theorem 4.19, Theorem 4.30]{TSI} and so is omitted.

\begin{theorem} \label{thm:monotonicity}
(Monotonicity) Given $\alpha\in I[a,b]$ and $f,g\in RDS_\alpha[a,b]$, if for every $x\in[a,b]$, $f(x)\ge g(x)$, then $$\int_{a}^{b}f(x)\, d \alpha\ge \int_{a}^{b}g(x)\, d \alpha.$$
\end{theorem}


\begin{theorem}\label{thm:linintegrator}
(Linearity of the Integrator) Given $\alpha, \beta\in BV[a,b]$ and $r_1, r_2\in\mathbb{R}$, define $\gamma = r_1\alpha+r_2\beta\in BV[a,b]$. If $f\in RDS_\alpha[a,b]$ and $f\in RDS_\beta[a,b]$, then $f\in RDS_\gamma[a,b],$ and $$\int_{a}^{b}f(x)\, d \gamma=r_1\int_{a}^{b}f(x)\, d \alpha+r_2\int_{a}^{b}f(x)\, d \beta.$$ 
\end{theorem}

The proof of Theorem \ref{thm:linintegrator} is an immediate consequence of Lemmas~\ref{lem:linintaddition} and~\ref{lem:linintscalar}.  The proofs of both lemmas are the same as that for the DS-integral~\cite[Theorem 4.30]{TSI}, and so omitted.

\begin{lemma}\label{lem:linintaddition}
Given $\alpha, \beta\in BV[a,b]$, define $\gamma = \alpha+\beta\in BV[a,b]$. If $f\in RDS_\alpha[a,b]$ and $f\in RDS_\beta[a,b]$, then $f\in RDS_\gamma[a,b],$ and $$\int_{a}^{b}f(x)\, d \gamma=\int_{a}^{b}f(x)\, d \alpha+\int_{a}^{b}f(x)\, d \beta.$$ 
\end{lemma}


\begin{lemma}\label{lem:linintscalar}
Given $\alpha\in BV[a,b]$ and $c\in\mathbb{R}$, define $\gamma = c\alpha\in BV[a,b]$. If $f\in RDS_\alpha[a,b]$, then $f\in RDS_\gamma[a,b],$ and $$\int_{a}^{b}f(x)\, d \gamma=c\int_{a}^{b}f(x)\, d \alpha.$$ 
\end{lemma}


\bigskip

We conclude this section with two lemmas, the first of which uses the RDS Criterion in its proof, and which we will use below.

\begin{lemma}\label{lem:fplus}
Given $\alpha \in BV[a,b]$ and $f\in RDS_\alpha[a,b]$, if we define $f^+(x)=\max(f(x), 0)$, then $f^+\in RDS_\alpha[a,b].$ Similarly, if we set $f^-(x) = \max(-f(x),0)$, then $f^-\in RDS_\alpha[a,b]$.
\end{lemma}

\begin{proof}
We will prove this for $f^+$; the proof for $f^-$ is essentially the same. First, consider $\alpha\in I[a,b]$. Fix $\epsilon>0$. By Theorem \ref{thm:RDScrit}, there exists step functions $u_f, v_f$ on some common partition  $\mathcal{P}=\{x_i\}_{i=0}^n$ of $[a,b]$ which bracket $f$ such that $$\int_a^b u_f(x)-v_f(x)\, d \alpha <\epsilon.$$ Let $u_f(x)=c_i$ and $v_f(x)=d_i$ for $x\in I_i$. Define $u,\,v$ by $$u(x)=\begin{cases}\max(c_i, 0), & x\in I_i,\\ f^+(x), & x= x_i,\end{cases}$$ and $$v(x)=\begin{cases}\max(d_i, 0), & x\in I_i,\\ f^+(x), & x= x_i.\end{cases}$$ Then, for $x\in I_i$, $$v(x)=\max(d_i, 0) \le \max(f(x), 0) \le \max(c_i, 0)= u(x)$$ and for $x=x_i$, $v(x)=f^+(x)=u(x)$. Moreover, if $\max(c_i, 0)=c_i$, then $$\max(c_i, 0)-\max(d_i, 0)=c_i-\max(d_i, 0)\le c_i-d_i$$ and if $\max(c_i, 0)=0$, then $$0\le \max(c_i, 0)-\max(d_i, 0)=-\max(d_i, 0)\le 0.$$ Hence, since $c_i-d_i\ge 0$, we have $$\max(c_i, 0)-\max(d_i, 0)\le \max(c_i-d_i, 0) = c_i-d_i.$$ Now by Definition \ref{def:stepint}, \begin{align*}
    \int_a^b u(x)-v(x)\, d \alpha &= \sum_{i=0}^n [u(x_i)-v(x_i)] \mu_{\alpha}(\{x_i\})+\sum_{i=1}^n(\max(c_i, 0)-\max(d_i, 0)) \mu_{\alpha}(I_i)\\
    &=\sum_{i=1}^n(\max(c_i, 0)-\max(d_i, 0)) \mu_{\alpha}(I_i)\\
    &\le \sum_{i=1}^n(c_i-d_i) \mu_{\alpha}(I_i)\\
    &\le \sum_{i=0}^n [u_f(x_i)-v_f(x_i)] \mu_{\alpha}(\{x_i\})+\sum_{i=1}^n(c_i-d_i) \mu_{\alpha}(I_i)\\
    &=\int_a^b u_f(x)-v_f(x)\, d \alpha\\
    &<\epsilon.
\end{align*}
So by Theorem \ref{thm:RDScrit}, $g\in RDS_\alpha[a,b]$.

\smallskip

Now, consider $\alpha\in BV[a,b]$. Then by Theorem \ref{thm:pnalpha}, $f\in RDS_{P\alpha}[a,b]$ and $f\in RDS_{N\alpha}[a,b]$. By the previous case, $f^+\in RDS_{P\alpha}[a,b]$ and $f^+\in RDS_{N\alpha}[a,b]$. Hence, by Theorem \ref{thm:pnalpha}, $f^+\in RDS_{\alpha}[a,b]$ as desired.
\end{proof}

\begin{remark}
Given $f\in B[a,b]$,  $f(x)= f^+(x)-f^-(x)$ and $|f(x)|=f^+(x)+f^-(x)$.
\end{remark}

The proof of the second lemma, which relies on Lemma~\ref{lem:fplus}, is identical to that for the DS-integral in~\cite[Theorem 4.34]{TSI}, and so is omitted.

\begin{lemma}\label{lem:triangleineq}
Given $\alpha \in BV[a,b]$ and $f\in RDS_\alpha[a,b]$, $|f|\in RDS_\alpha[a,b]$ and $f\in RDS_{V\alpha}[a,b]$, and $$\left|\int_a^bf(x)\, d \alpha\right|\le \int_a^b|f(x)|\, d V\alpha.$$
\end{lemma}



\section{Convergence Theorems}
\label{section:convergence}

In this section, we  begin to consider limits and the RDS-integral.  Our ultimate goal is to prove the bounded convergence theorem, but we defer that proof to Section~\ref{section:BCT} as we will need results from Sections~\ref{section:convergence} and~\ref{section:integrability} to prove it.  Here we prove that the limit and the integral can be exchanged in two cases:  first when the integrands converge uniformly, and  second when the integrators converge in bounded variation norm. 

\begin{theorem}\label{thm:uniformconv}
Given $\alpha\in BV[a,b]$ and a sequence $\{f_k\}_{k=1}^\infty$ of bounded functions on $[a,b]$, suppose $f_k\in RDS_\alpha[a,b]$ for each $k\in \mathbb{N}$ and there exists a function $f$ on $[a,b]$ such that $f_k\to f$ uniformly. Then, $f\in RDS_\alpha [a,b]$ and 
\begin{equation} \label{eqn:uniformconv1}
\lim_{k\to \infty} \int_a^b f_k(x)\, d \alpha = \int_a^b f(x)\, d \alpha.
\end{equation}
\end{theorem}

\begin{proof}
First, we consider the case when $\alpha\in I[a,b]$. Suppose $\mu_\alpha([a,b])=\alpha(b)-\alpha(a)=0$. Then, $\alpha$ is constant. By Lemma \ref{lem:constint}, $f\in RDS_\alpha[a,b]$ and for each $k\in\mathbb{N},$ $$\int_a^bf_k(x)\, d \alpha = 0 = \int_a^bf(x)\, d \alpha.$$ Hence, \eqref{eqn:uniformconv1} holds.

Now suppose $\mu_\alpha([a,b])=\alpha(b)-\alpha(a)>0.$ Fix $\epsilon>0$. By uniform convergence, there exists $K\in \mathbb{N}$ such that for all $x\in [a,b]$, $$|f_K(x)-f(x)|<\frac{\epsilon}{3(\alpha(b)-\alpha(a))}.$$ Since $f_K \in RDS_\alpha[a,b]$, by Theorem \ref{thm:RDScrit}, there exists $u_K, v_K$ step functions with respect to a common partition $\{x_i\}_{i=0}^n$ which bracket $f_K$ such that $$\int_a^b u_K(x)-v_K(x) \, d \alpha < \frac{\epsilon}{3}.$$ Let $u_K(x)=c_i$ and $v_K(x)=d_i$ for $x\in I_i$.

Define step functions $u,\,v$ as $$u(x) = u_K(x)+\frac{\epsilon}{3(\alpha(b)-\alpha(a))}$$ and $$v(x) = v_K(x)-\frac{\epsilon}{3(\alpha(b)-\alpha(a))}$$ Hence, for $x\in [a,b]$, $$f(x) < f_K(x) + \frac{\epsilon}{3(\alpha(b)-\alpha(a))} \le u_K(x)+\frac{\epsilon}{3(\alpha(b)-\alpha(a))}= u(x)$$ and $$f(x) > f_K(x) - \frac{\epsilon}{3(\alpha(b)-\alpha(a))} \ge v_K(x)-\frac{\epsilon}{3(\alpha(b)-\alpha(a))}= v(x).$$ Moreover, \begin{align*}
    \int_a^b u(x)-v(x)\, d \alpha &= \sum_{i=0}^n [u(x_i)-v(x_i)] \mu_{\alpha}(\{x_i\})\\
    &\quad \quad+\sum_{i=1}^n(c_i+\frac{\epsilon}{3(\alpha(b)-\alpha(a))}-d_i+\frac{\epsilon}{3(\alpha(b)-\alpha(a))}) \mu_{\alpha}(I_i)\\
    &=\sum_{i=0}^n [u_K(x_i)-v_K(x_i)] \mu_{\alpha}(\{x_i\})+\sum_{i=1}^n(c_i-d_i)\mu_\alpha(I_i)\\ &\quad \quad +\frac{2\epsilon}{3(\alpha(b)-\alpha(a))}\left[\sum_{i=0}^n \mu_{\alpha}(\{x_i\})+\sum_{i=1}^n\mu_{\alpha}(I_i)\right]\\
    &=\int_a^b u_K(x)-v_K(x)\, d \alpha + \frac{2\epsilon}{3(\alpha(b)-\alpha(a))}\int_a^b 1\, d \alpha\\
    &<\frac{\epsilon}{3}+\frac{2\epsilon}{3(\alpha(b)-\alpha(a))}(\alpha(b)-\alpha(a))\\
    &=\epsilon.
\end{align*}

\noindent So by Theorem \ref{thm:RDScrit}, $f\in RDS_\alpha [a,b].$
\newline

To show that \eqref{eqn:uniformconv1} holds, fix $\epsilon>0$. Then, there exists $K$ such that for all $k\ge K^\prime$, 
$$|f_k(x)-f(x)|<\frac{\epsilon}{2(\alpha(b)-\alpha(a))}.$$ 
Hence, for $k\ge K^\prime$, 
\begin{multline*}
    \left|\int_a^b f_k(x)\, d \alpha-\int_a^b f(x)\, d \alpha\right| 
    = \left|\int_a^b f_k(x)-f(x)\, d \alpha\right|
    \le \int_a^b |f_k(x)-f(x)|\, d \alpha \\
    \le \int_a^b \frac{\epsilon}{2(\alpha(b)-\alpha(a))}\, d \alpha
    =\frac{\epsilon}{2(\alpha(b)-\alpha(a))}(\alpha(b)-\alpha(a))
    =\frac{\epsilon}{2}<\epsilon.
\end{multline*}

Now consider $\alpha\in BV[a,b]$. Then by Theorem \ref{thm:pnalpha}, for each $k\in \mathbb{N}$, $f_k\in RDS_{P\alpha}[a,b]$ and $f_k\in RDS_{N\alpha}[a,b]$. By the previous case, $f\in RDS_{P\alpha}[a,b]$ and $f\in RDS_{N\alpha}[a,b]$. Thus, $f\in RDS_\alpha[a,b]$. Moreover, 
\begin{multline*}
    \lim_{k\to \infty} \int_a^b f_k(x)\, d \alpha
    =\lim_{k\to \infty} \left[\int_a^b f_k(x)\, d P\alpha-\int_a^b f_k(x)\, d N\alpha\right]\\
    =\int_a^b f(x)\, d P\alpha-\int_a^b f(x)\, d N\alpha
    =\int_a^b f(x)\, d \alpha.
\end{multline*} 
\end{proof}

\medskip

To state our result on convergence of the integrators, we first give the definition of the bounded variation norm, following~\cite[Proposition 3.43]{TSI}; recall the definition of the variation of a BV function in Definition~\ref{defn:BV}.

\begin{definition}
    The function $\|\cdot\|_{BV}: BV[a,b]\to [0, \infty)$ defined for $\alpha\in BV[a,b]$ by $$\|\alpha\|_{BV}= \sup_{x\in [a,b]}|\alpha(x)|+ V\alpha(b).$$ 
\end{definition}

\begin{remark}
For a proof that $BV[a,b]$ is a normed space with respect to this norm, and for properties of this space, see~\cite[Section~3.6]{TSI}.
\end{remark}

\begin{theorem}\label{thm:convintegrator}
Given a sequence of functions $\{\alpha_k\}_{k=1}^\infty$, suppose for each $k\in \mathbb{N}$, $\alpha_k\in BV[a,b]$ and $\alpha_k\to \alpha$ in BV norm. Moreover, suppose  $f\in RDS_{\alpha_k}[a,b]$ for all $k\in \mathbb{N}$. Then, $f\in RDS_\alpha[a,b]$ and 
\begin{equation} \label{eqn:convintegrator1}
\lim_{k\to\infty}\int_a^bf(x)\, d \alpha_k=\int_a^bf(x)\, d \alpha.
\end{equation}
\end{theorem}

\begin{proof}
 First, we consider the case where all $\alpha_k\in I[a,b]$. Fix $\epsilon>0$. Note $f\in B[a,b]$ so there exists $M>0$ such that $|f(x)|\le M$ for all $x\in [a,b]$. Further, there exists $K\in\mathbb{N}$ such that $$||\alpha-\alpha_k||_{BV}<\frac{\epsilon}{3M}.$$ Since $f\in RDS_{\alpha_k}[a,b]$, by Corollary \ref{cor:RDScritbdd}, there exist step functions $u_K, v_K$ such that $$-M\le v_K(x)\le f(x)\le u_K(x)\le M$$ and $$\int_a^b u_K(x)-v_K(x)\, d \alpha_k < \frac{\epsilon}{3}.$$ Suppose $u_K, v_K$ are defined with respect to the common partition $\{x_i\}_{i=0}^n$ where $u_K(x)=c_i, v_K(x)=d_i$ for $x\in I_i$. Define step functions $u,\,v$ by $$u(x)=\begin{cases}u_K(x), & x\in I_i,\\f(x_i), & x=x_i,\\\end{cases} \quad \text{ and }\quad v(x)=\begin{cases}v_K(x), & x\in I_i,\\f(x_i), & x=x_i.\\\end{cases}$$ So, \begin{align*}
     \int_a^b u(x)-v(x)\, d \alpha_k &= \sum_{i=0}^n(u(x_i)-v(x_i))\mu_{\alpha_k}(\{x_i\})+\sum_{i=1}^n(c_i-d_i)\mu_{\alpha_k}(I_i)\\
     &=\sum_{i=1}^n(c_i-d_i)\mu_{\alpha_k}(I_i)\\
     &\le \sum_{i=0}^n(u_K(x_i)-v_K(x_i))\mu_{\alpha_k}(\{x_i\})+\sum_{i=1}^n(c_i-d_i)\mu_{\alpha_k}(I_i)\\
     &=\int_a^b u_K(x)-v_K(x)\, d \alpha_k\\
     &<\frac{\epsilon}{3}.
 \end{align*} 
 
 Then by Lemma \ref{lem:triangleineq}, \begin{align*}
     \int_a^b u(x)-v(x)\, d \alpha &= \int_a^b u(x)-v(x)\, d \alpha_k+\int_a^b u(x)-v(x)\, d (\alpha-\alpha_k)\\
     &\le \int_a^b u(x)-v(x)\, d \alpha_k+\left|\int_a^b u(x)-v(x)\, d (\alpha-\alpha_k)\right|\\
     &<\frac{\epsilon}{3}+\int_a^b|u(x)-v(x)|\, d V(\alpha-\alpha_k)\\
     &\le \frac{\epsilon}{3}+2M\int_a^b\, d V(\alpha-\alpha_k)\\
     &= \frac{\epsilon}{3}+2MV(\alpha-\alpha_k)(b)\\
     &\le \frac{\epsilon}{3}+2M\|\alpha-\alpha_k\|_{BV}\\
     &<\frac{\epsilon}{3}+2M\frac{\epsilon}{3M}\\
     &=\epsilon.
 \end{align*}
 Therefore, by Theorem \ref{thm:RDScrit}, $f\in RDS_\alpha[a,b]$.

To show that \eqref{eqn:convintegrator1} holds, fix $\epsilon>0$. Since $\alpha_k \to \alpha$ in BV norm, there exists $K^\prime$ such that for all $k\ge K^\prime$, 
$$\|\alpha_k-\alpha\|_{BV}<\frac{\epsilon}{M}.$$ 
Then, for all $k\ge K^\prime$, 
\begin{multline*}
    \left|\int_a^b f(x)\, d \alpha_k-\int_a^b f(x)\, d \alpha\right| 
    = \left|\int_a^b f(x)\, d (\alpha_k-\alpha)\right|
    \le \int_a^b |f(x)|\, d V(\alpha_k-\alpha) \\
    \le M\int_a^b dV(\alpha_k-\alpha)
    =MV(\alpha_k-\alpha)(b)
    \le M\|\alpha_k-\alpha\|_{BV}
    <\epsilon.
\end{multline*}
 
 Now, consider $\alpha_k\in BV[a,b]$. Then by Theorem \ref{thm:pnalpha}, $f\in RDS_{P\alpha_k}[a,b]$ and $f\in RDS_{N\alpha_k}[a,b]$ for every $k\in\mathbb{N}$. By \cite[Proposition 3.47]{TSI}, $P\alpha_k \to P\alpha, N\alpha_k\to N\alpha$ in $BV$ norm. Thus, $f\in RDS_{P\alpha}[a,b]$ and $f\in RDS_{N\alpha}[a,b]$ by the previous case. Therefore, $f\in RDS_\alpha[a,b]$. Moreover, 
 \begin{multline*}
     \lim_{k\to \infty} \int_a^b f(x) \, d \alpha_k 
     = \lim_{k\to \infty} \left[\int_a^b f(x) \, d P\alpha_k -\int_a^b f(x) \, d N\alpha_k\right] \\
     = \int_a^b f(x) \, d P\alpha -\int_a^b f(x) \, d N\alpha
     =\int_a^b f(x) \, d \alpha.
 \end{multline*}
\end{proof}

\medskip

Finally, we note that in Theorem~\ref{thm:convintegrator}, convergence in BV norm is necessary, in the sense that it cannot be replaced with the weaker condition of uniform convergence (which is implied by convergence in norm).  This was shown for the interior DS-integral in~\cite[Example~4.52]{TSI}.  In this example all the integrators are continuous, so by Theorem~\ref{thm:RDSiffDSintegrator} we conclude that the same example holds for the RDS-integral as well.

\section{RDS Integrability Theorems}
\label{section:integrability}

In this section we turn to the following problem:  given $\alpha \in BV[a,b]$, characterize the functions $f\in B[a,b]$ such that $f\in RDS_\alpha[a,b]$.  Our goal is to prove a version of the Lebesgue criterion for the RDS-integral, Theorem~\ref{thm:RDSLebcrit}; this will let us completely characterize the $f\in RDS_\alpha[a,b]$ in terms of the discontinuities of $f$ and $\alpha$.   To accomplish this, we need to work with the so-called saltus decomposition of functions in $BV[a,b]$: this will allow us to study the behavior of the RDS-integral at a discontinuity of $\alpha$ by considering a Heaviside function that corresponds to this discontinuity.  

Finally, as an application of the results in this section, we prove a generalized version of integration by parts for the RDS-integral.   For the classical definitions of the RS or DS-integrals, integration by parts can be stated as an integrability theorem:  see, for example,~\cite[Theorem~29.7]{MR0393369}. While we do not prove such a version here, it still seemed natural to include integration by parts in this section. 

\medskip

We begin with a classical result that also holds for all  definitions of the Stieltjes integral.

\begin{theorem}\label{thm:continuousinRDS}
Suppose $\alpha\in BV[a,b]$ and $f\in C[a,b]$. Then, $f\in RDS_\alpha[a,b]$.
\end{theorem}

\begin{proof}
We will first consider the case when $\alpha\in I[a,b]$. The proof is very similar to the proof of Theorem~\ref{thm:RDSiffDSintegrand}.  Suppose $\mu_\alpha([a,b])= \alpha(b)-\alpha(a) = 0$. Then, $\alpha$ is constant, and so by Lemma \ref{lem:constint}, $f\in RDS_\alpha[a,b]$.

Now suppose $\mu_\alpha([a,b])= \alpha(b)-\alpha(a) >0$. Fix $\epsilon>0$. Since $f$ is continuous on the closed interval $[a,b]$, it must also be uniformly continuous on $[a,b]$. Fix $\delta$, $N$ and the regular partition $\mathcal{P}_N=\{x_i\}_{i=1}^N=\{a+\frac{b-a}{N}i\}_{i=0}^N$.  Define step functions $u,\,v$ as in~\eqref{eqn:RDSiffDSintegrand1}.  Then $u,\,v$ bracket $f$ on $[a,b]$, and, since the partition intervals have lengths less than $\delta$, if $x\in I_i$, then 
$$u(x)-v(x) =M_i- m_i = f(x_i^M)-f(x_i^m)= |f(x_i^M)-f(x_i^m)| < \frac{\epsilon}{\alpha(b)-\alpha(a)}, $$ 
and if $x=x_i$, $u(x)-v(x)=0$. Hence, by Definition \ref{def:stepint}, 
\begin{align*}
    \int_a^b u(x)-v(x)\, d \alpha &= \sum_{i=0}^N [u(x_i)-v(x_i)] \mu_{\alpha}(\{x_i\})+\sum_{i=1}^N[M_i-m_i] \mu_{\alpha}(I_i)\\
    &<\sum_{i=1}^N\frac{\epsilon}{\alpha(b)-\alpha(a)}\mu_{\alpha}(I_i)\\
    &=\frac{\epsilon}{\alpha(b)-\alpha(a)}\sum_{i=1}^N\mu_{\alpha}(I_i)\\
    &\le \frac{\epsilon}{\alpha(b)-\alpha(a)}\left[\sum_{i=0}^N\mu_\alpha(\{x_i\})+\sum_{i=1}^N\mu_{\alpha}(I_i)\right]\\
    &=\frac{\epsilon}{\alpha(b)-\alpha(a)}\int_a^b 1\, d \alpha\\
    &=\epsilon.
\end{align*}
Therefore, by Theorem \ref{thm:RDScrit}, $f\in RDS_\alpha[a,b]$.
\newline

Finally, if  $\alpha\in BV[a,b]$, then by the previous case, $f\in RDS_{P\alpha}[a,b]$ and $f\in RDS_{N\alpha}[a,b]$. Thus, by Theorem \ref{thm:pnalpha}, $f\in RDS_\alpha[a,b]$.

\end{proof}

\medskip

In our next result we show that if the integrand is a function of bounded variation, then it is integrable with respect to any integrator in $BV[a,b]$.  Already with this result, we see that the RDS-integral is much more general than the classical definitions of the Stieltjes integral:  the RS-integral does not exist if $f$ and $\alpha$ have a common discontinuity (see \cite[Theorem~5.31]{TSI}), and the classical exterior DS-integral only exists in certain cases (see~\cite[Theorem~4.13]{Nielson}).  This result also holds for the interior DS-integral (see~\cite[Theorem~4.38]{TSI}) and  was one of the motivations for using that definition.

To prove this result  we introduce the saltus decomposition which decomposes a function of bounded variation  into a continuous piece and an infinite sum of Heaviside functions.  This decomposition is developed in detail in~\cite[Section~3.5]{TSI}; here we refine that definition slightly to meet our needs.

\begin{definition}\label{def:saltusset}
    A saltus set in $[a,b]$ is a sequence of ordered pairs $\{(x_i, a_i)\}$ such that the points $x_i$ are a collection of disjoint points in $[a,b]$, and the $a_i\in \mathbb{R}$, $a_i\neq 0$, are such that 
    $$\sum_{i=1}^\infty |a_i|<\infty.$$ 
    Associated with any saltus set is a saltus function,
    $$f(x) = \sum_{i=1}^\infty a_i H_c(x-x_i),$$
    where  $c\in \mathbb{R}$. If $c=0$,  $f$ is a left continuous saltus function, and if $c=1$, it is a right continuous saltus function.  A saltus function is reduced if when $c=1$, $x_i\neq a$ for all $i\in \mathbb{N}$, or if when $c=0$, $x_i\neq b$ for all $i\in \mathbb{N}$.
\end{definition}

\begin{prop}\label{prop:ReducedSaltusfunctions}
    Given $f\in BV[a,b]$, there exists a saltus decomposition into functions of bounded variation
    \begin{equation} \label{eqn:reducedSaltusfunctions1}
f= Gf+S_Lf+S_Rf,
\end{equation}
    where $Gf$ is continuous,  $S_Lf$ is a reduced, left continuous saltus function,  and $S_Rf$ is a reduced, right continuous saltus function.
\end{prop}

\begin{proof}
    Let $\{x_i\}_{i=1}^\infty$ be the set of discontinuities of $f$. By \cite[Theorem 3.38]{TSI}, we have a saltus decomposition 
    \eqref{eqn:reducedSaltusfunctions1},
    where $Gf$ is a continuous function of bounded variation, and 
    \[ S_Lf = \sum_{i=1}^\infty \big(f(x_i+)-f(x_i)\big) H_0(x-x_i) \;
    \text{and} 
    \; S_Rf = \sum_{i=1}^\infty \big(f(x_i)-f(x_i-)\big) H_1(x-x_i)\]
    are left and right continuous saltus functions (which are functions of bounded variation:  see \cite[Proposition~3.37]{TSI}).  
    It follows from this definition (also see \cite[Remark~3.42]{TSI}) that we may assume that $b$ is not in the saltus set of $S_Lf$ since $f(b+)-f(b)=0$ so even if $f$ is discontinuous at $b$, the corresponding term in the definition of $S_Lf$ is $0$.  Similarly, we may assume that $a$ is not in the saltus set of $S_Rf$.   Thus, both $S_Lf$ and $S_Rf$ are reduced saltus functions.
\end{proof}

\begin{remark}
The reduced saltus functions $S_Lf$ and $S_Rf$ satisfy $S_Lf(a)=S_Rf(a)=0$.
\end{remark}

\begin{theorem}\label{thm:BVinRDS}
Suppose $\alpha\in BV[a,b]$ and $f\in BV[a,b]$. Then, $f\in RDS_\alpha[a,b]$.
\end{theorem}

\begin{proof}
By Proposition~\ref{prop:ReducedSaltusfunctions}, form the saltus decomposition of $f$.  By Theorem \ref{thm:continuousinRDS}, $Gf\in RDS_{\alpha}[a,b]$. By \cite[Lemma 3.36]{TSI}, there exists  a sequence of step functions $\{S_L^nf\}_{n=1}^\infty$ such that $S_L^nf \to S_L f$ uniformly. Since these are step functions, $S_L^n f\in RDS_\alpha[a,b]$ by Proposition~\ref{prop:stepagree}. Thus, by Theorem~\ref{thm:uniformconv}, $S_Lf\in RDS_\alpha[a,b]$. Similarly, $S_Rf\in RDS_\alpha[a,b]$. Therefore, by Theorem~\ref{thm:linearity},  $f\in RDS_\alpha[a,b].$ 
\end{proof}

\medskip

We now consider the RDS-integral when the integrand $\alpha$ is a saltus function.  We will show that every bounded function is RDS-integrable when the integrator is a saltus function that arises in the saltus decomposition of a function of bounded variation.  As a first step we give two key results for integrators that are Heaviside functions.  These results capture the fact that such integrators act like Lebesgue-Stieltjes measures that are point masses.  This behavior highlights where our definition improves on the historic ones.  (Compare this result to the behavior of the interior DS-integral in~\cite[Lemma~5.6]{TSI}.)

\begin{prop}\label{prop:BddinRDSH_c}
Suppose $\alpha\in BV[a,b]$ is the Heaviside function $\alpha(x)=H_c(x-x_0)$ for some $x_0\in (a,b), c\in \mathbb{R}$  and $f\in B[a,b]$. Then, $f\in RDS_\alpha[a,b]$ and 
\begin{equation} \label{eqn:BddinRDSH_c}
\int_a^b f(x)\, d \alpha = f(x_0).
\end{equation}
\end{prop}

\begin{proof}
Since $f\in B[a,b]$, there exists $M>0$ such that $|f(x)|\le M$ for every $x\in [a,b]$. Define step functions $u,\,v$ by 
$$u(x) = \begin{cases}M, & x\in [a, b]\backslash \{x_0\},\\ f(x_0), & x=x_0, \end{cases} \quad \text{and} \quad 
v(x) = \begin{cases}-M, & x\in [a, b]\backslash \{x_0\},\\ f(x_0), & x=x_0. \end{cases}$$ 
Then $v(x)\le f(x)\le u(x)$ and, by Definition \ref{def:stepint},
\begin{align*}
    \int_{a}^b u(x)\, d \alpha &= u(a)\mu_\alpha(\{a\})+u(x_0)\mu_\alpha(\{x_0\})+u(b)\mu_\alpha(\{b\})\\&\quad \quad+M\cdot\mu_\alpha((a, x_0))+M\cdot\mu_\alpha((x_0, b))\\
    &= M\cdot 0 + f(x_0)\cdot 1+M\cdot 0 + M\cdot 0+M\cdot 0\\
    &=f(x_0).
\end{align*}
Similarly,
\begin{align*}
    \int_{a}^b v(x)\, d \alpha &= v(a)\mu_\alpha(\{a\})+v(x_0)\mu_\alpha(\{x_0\})+v(b)\mu_\alpha(\{b\})\\&\quad \quad-M\cdot\mu_\alpha((a, x_0))-M\cdot\mu_\alpha((x_0, 1))\\
    &= -M\cdot 0 + f(x_0)\cdot 1-M\cdot 0 -M\cdot 0-M\cdot 0\\
    &=f(x_0).
\end{align*}
Thus by Definition \ref{def:RDSint} and Proposition \ref{prop:lowerupperineq}, $$f(x_0)\le L_\alpha(f, [a,b]) \le U_\alpha(f, [a, b])\le f(x_0),$$ which proves that $f\in RDS_\alpha[a, b]$ and \eqref{eqn:BddinRDSH_c} holds.
\end{proof}

In Proposition~\ref{prop:BddinRDSH_c} we assumed that the discontinuity of the Heaviside function was in the interior of the interval $(a,b)$. We now consider when the discontinuity lies at $x_0=a$ or $x_0=b$. If $c=1$ and $x_0=a$ then $H_c(x-x_0)=1$, a constant function on $[a,b]$ and so, the result would not be true by Lemma \ref{lem:constint}. A similar situation happens if $c=0$ and $x_0=b$. Thus, we consider only the combinations $c=0, x_0=a$ or $c=1, x_0=b$.

\begin{prop}\label{prop:BddinRDSH_cExtraCases}
Suppose $\alpha\in BV[a,b]$ is the Heaviside function $\alpha(x)=H_c(x-x_0)$ with either $x_0=a, c=0$ or $x_0=b, c=1$. If $f\in B[a,b]$, then $f\in RDS_\alpha[a,b]$ and $$\int_a^b f(x)\, d \alpha = f(x_0).$$
\end{prop}

\begin{proof}
Suppose $x_0=a$ and $c=0$.  Define $u,\,v$  as in the proof of Proposition \ref{prop:BddinRDSH_c}. Then, \begin{align*}
    \int_{a}^b u(x)\, d \alpha &= u(a)\mu_\alpha(\{a\})+u(b)\mu_\alpha(\{b\})+M\cdot\mu_\alpha((a, b))\\
    &= f(a)\cdot (1-c) + M\cdot 0+ M\cdot 0\\
    &= f(a),
\end{align*}

and \begin{align*}
    \int_{a}^b v(x)\, d \alpha &= v(a)\mu_\alpha(\{a\})+v(b)\mu_\alpha(\{b\})-M\cdot\mu_\alpha((a, b))\\
    &= f(a)\cdot (1-c) - M\cdot 0- M\cdot 0\\
    &= f(a).
\end{align*}

Thus, as in the proof of Proposition \ref{prop:BddinRDSH_c}, we are done. The other case $x_0=b, c=1$, follows similarly.
\end{proof}

\begin{theorem}\label{thm:BddinRDSsaltus}
Suppose $\alpha\in BV[a,b]$ is a reduced saltus function of the form 
$$\alpha(x)=\sum_{i=1}^\infty a_iH_c(x-x_i).$$  If $f\in B[a,b]$, then $f\in RDS_\alpha[a,b]$ and 
$$\int_a^b f(x)\, d \alpha = \sum_{i=1}^\infty a_if(x_i).$$
\end{theorem}

\begin{proof}
Let $\alpha_n$ denote the partial sums of $\alpha$: that is, 
$$\alpha_n(x)=\sum_{i=1}^n a_iH_c(x-x_i).$$ 
By \cite[Proposition 3.48]{TSI}, $\alpha_n\to \alpha$ in BV norm. Moreover, for every $n\in \mathbb{N}$,  $\alpha_n$ is a linear combination of Heaviside functions. Therefore, by Propositions~\ref{prop:BddinRDSH_c} and~\ref{prop:BddinRDSH_cExtraCases}, and Theorem \ref{thm:linearity}, $f\in RDS_{\alpha_n}[a,b]$ for every $n\in\mathbb{N}$. Hence, by Theorem \ref{thm:convintegrator}, $f\in RDS_\alpha[a,b]$ and 
\begin{multline*}
    \int_a^b f(x)\, d \alpha 
    = \lim_{n\to \infty} \int_a^b f(x)\, d \alpha_n\\
    = \lim_{n\to \infty} \sum_{i=1}^n a_i\int_a^b f(x)\, d H_c(x-x_i)
    = \lim_{n\to \infty}\sum_{i=1}^n a_if(x_i)
    = \sum_{i=1}^\infty a_i f(x_i).
\end{multline*}

\end{proof}

We have two immediate corollaries of Theorem~\ref{thm:BddinRDSsaltus}. First, we have that the Dirichlet function,
$$D(x)=\begin{cases}
        1,& x\in \mathbb{Q},\\
        0, & x\not \in \mathbb{Q},
    \end{cases}$$
is always integrable against a saltus function.  This stands in contrast to both the interior and exterior DS-integrals. 

\begin{corollary}\label{cor:Dirichlet}
    Let $\alpha\in BV[a,b]$ be a reduced saltus function.  Then $D\in RDS_{\alpha}[a,b]$.
\end{corollary}


Second, we can show that the values of $\alpha$ at its discontinuities in the interior of $[a,b]$ do not affect the integral, so we can normalize the integrand by assuming it is left or right continuous.

\begin{corollary}\label{cor:WLOGleftcontinuity}
    Given $\alpha\in BV[a,b]$, let $\{x_i\}_{i=1}^\infty$ be its set of discontinuities and suppose $\alpha$ is not discontinuous at $b$.  Suppose $\beta\in BV[a,b]$ is such that $\alpha(x)=\beta(x)$ for all $x\neq x_i$ and $\beta(x_i)=\beta(x_i-)$ for every $i$. (That is, $\beta$ is left continuous.)  Then, $f\in RDS_\alpha[a,b]$ if and only if $f\in RDS_\beta [a,b]$ and in this case, 
    \begin{equation} \label{eqn:wlog1}
    \int_a^b f(x)\, d\alpha = \int_a^b f(x)\, d\beta.
    \end{equation}
    The same is true if we assume $\alpha$ is not discontinuous at $a$ and we take $\beta$ to be right continuous.
\end{corollary}

\begin{proof}
We will prove this for $\beta$ left continuous; the proof if it is right continuous is essentially the same.
    Let $c_i = \alpha(x_i)-\alpha(x_i-)$.  Then by Proposition~\ref{prop:ReducedSaltusfunctions}, we have that $\{(x_i, c_i) : c_i \neq 0\}$ forms the saltus set for $S_R\alpha$.  Since $S_R\alpha$ is reduced, we have that $a\not\in \{x_i\}_{i=1}^\infty$. by assumption, $b\not\in \{x_i\}_{i=1}^\infty$.  
    
    Since $\alpha(x_i-)=\beta(x_i-)=\beta(x_i)$,
    $$\alpha(x)-\beta(x) = \begin{cases}
        0, & x\neq x_i,\\
        c_i, & x=x_i.
    \end{cases}$$ 
    By  Definition \ref{def:saltusset}, 
    $$\sum_{i=1}^\infty |c_i|<\infty,$$
    so we can rewrite this as
    $$\alpha(x)-\beta(x) = \sum_{i=1}^\infty c_iH_1(x-x_i)-\sum_{i=1}^\infty c_iH_0(x-x_i).$$ 
    The righthand side  is the difference of two reduced saltus functions. Therefore, by Theorems~\ref{thm:BddinRDSsaltus} and~\ref{thm:linintegrator}, for any $f\in B[a,b]$, we have $f\in RDS_{\alpha-\beta}[a,b]$. 

    To complete the proof, first suppose $f\in RDS_\alpha[a,b]$. Then, by Theorem \ref{thm:linintegrator}, since $\beta = -(\alpha-\beta)+ \alpha$, we have $f\in RDS_\beta[a,b]$. Alternatively, suppose $f\in RDS_\beta[a,b]$. Again, by Theorem \ref{thm:linintegrator}, we have $f\in RDS_\alpha[a,b]$.  To see that \eqref{eqn:wlog1} holds, note that  by  Theorems~\ref{thm:BddinRDSsaltus} and~\ref{thm:linintegrator}, 
    $$\int_a^b f(x)\, d\alpha - \int_a^b f(x)\, d\beta
    =\int_a^b f(x)\, d(\alpha-\beta)= \sum_{i=1}^\infty c_if(x_i)- \sum_{i=1}^\infty c_if(x_i)
    =0.$$ 
\end{proof}
\medskip

We now give a complete characterization of those functions $f$ which are integrable with respect to a given function of bounded variation  $\alpha$.  To put our result into context, recall that the Lebesgue criterion (see \cite[Theorem~2.2]{TSI}) states that a bounded function $f$ is Darboux integrable if and only if  the set of discontinuities of $f$ has Lebesgue measure $0$.  Lebesgue measure corresponds to the Lebesgue-Stieltjes measure induced by the continuous integrator $\alpha(x)=x$.  Our result is the exact analogue of this, and does not require the additional hypotheses needed for either the interior or exterior DS-integral.  To state it we need a definition.

\begin{definition}\label{def:alphameasure0}
Given $\alpha\in I[a,b]$, a set $E\subset [a,b]$ has $\alpha$-measure $0$ if for every $\epsilon>0$, there exists a collection $\{I_i\}_{i=1}^\infty$ of open intervals such that 
$$E\subset \bigcup_{i=1}^\infty I_i \quad \text{ and } \quad 
\sum_{i=1}^\infty \mu_{\alpha}(I_i)<\epsilon.$$
\end{definition}

\begin{remark} \label{remark:two-defns}
When $\alpha$ is continuous, this definition of $\alpha$-measure $0$ agrees with \cite[Definition 5.17]{TSI}.  If $\alpha$ has a non-trivial saltus component, Definition~\ref{def:alphameasure0} detects point masses, while the definition in \cite{TSI} does not.
\end{remark}

\begin{theorem}\label{thm:RDSLebcrit}
(RDS-Lebesgue Criterion) Given $f\in B[a,b]$ and $\alpha\in BV[a,b]$, $f\in RDS_\alpha[a,b]$ if and only if the subset of $[a,b]$ where $f$ is discontinuous has $G(P\alpha)$-measure $0$ and $G(N\alpha)$-measure $0$ where $G(P\alpha), G(N\alpha)$ are the continuous pieces of the saltus decomposition of $P\alpha, N\alpha$ respectively.
\end{theorem}

\begin{remark}
    An equivalent version of Theorem~\ref{thm:RDSLebcrit}, with a very different proof couched in terms of his definition of a generalized RS-integral, was given by Ter Horst~\cite[Theorem~D]{LebStieltjes}.
\end{remark}

The following lemma let us  takes advantage of the saltus decomposition to  reduce to only considering the $G\alpha$-measure of the discontinuities of $f$.

\begin{lemma}\label{lem:RDSGalpha}
Given $\alpha\in BV[a,b]$ and $f\in B[a,b]$, $f\in RDS_\alpha[a,b]$ if and only if $f\in RDS_{G\alpha}[a,b]$.
\end{lemma}

\begin{proof}
    Suppose first that $f\in RDS_\alpha[a,b]$. By Theorem \ref{thm:BddinRDSsaltus}, $f\in RDS_{S_L\alpha}[a,b]$ and $f\in RDS_{S_R\alpha}[a,b]$. Since $G\alpha =\alpha - S_L\alpha - S_R\alpha$, by Theorem \ref{thm:linintegrator}, $f\in RDS_{G\alpha}[a,b]$. Conversely, suppose $f\in RDS_{G\alpha}[a,b]$. Since  $f\in RDS_{S_L\alpha}[a,b]$ and $f\in RDS_{S_R\alpha}[a,b]$, again by Theorem \ref{thm:linintegrator} we have that $f\in RDS_{\alpha}[a,b]$.
\end{proof}

\smallskip

\begin{proof}[Proof of Theorem \ref{thm:RDSLebcrit}]
We first suppose that $\alpha\in I[a,b]$.
 Let $D\subset [a,b]$ be the of discontinuities of $f$.  We will show $f\in RDS_\alpha [a,b]$ if and only if $D$ has $G\alpha$-measure $0$.  If $f\in RDS_{\alpha}[a,b]$, then by Lemma \ref{lem:RDSGalpha}, $f\in RDS_{G\alpha}[a,b]$.
 Since $G\alpha$ is continuous, by Theorem \ref{thm:RDSiffDSintegrand}, $f\in DS_{G\alpha}[a,b]$.  We can now use the characterization of the DS-integral in~\cite[Proposition 5.25]{TSI}, which gives that this is the case if and only if $D$ has $G\alpha$-measure $0$.  (Note by Remark~\ref{remark:two-defns}, the definitions of $G\alpha$ agree.)  
 
Conversely, suppose $D$ has $G\alpha$-measure $0$. By \cite[Proposition 5.25]{TSI}, $f\in RDS_{G\alpha}[a,b]$. By Theorem \ref{thm:BddinRDSsaltus}, $f\in RDS_{S_L\alpha}[a,b]$ and $f\in RDS_{S_R\alpha}[a,b]$. Hence, by Theorem \ref{thm:linintegrator}, $f\in RDS_\alpha[a,b]$.

Now suppose that  $\alpha\in BV[a,b]$. If $f\in RDS_\alpha[a,b]$, then by Theorem \ref{thm:pnalpha}, $f\in RDS_{P\alpha}[a,b]$ and $f\in RDS_{N\alpha}[a,b]$. From the previous case, $D$ has $G(P\alpha)$-measure $0$ and $G(N\alpha)$-measure $0$. The converse follows since this argument is reversible. 
\end{proof}

\medskip

We conclude this section by proving a generalization of integration by parts.   As with the DS-integral (see~\cite[Theorem 5.3]{TSI}), the classical formula for integration by parts need not hold.  Fix $c,\,d\in \mathbb{R}$ and $t\in (a,b)$; then we have that (see Lemma~\ref{lem:H_c})
\[ \int_a^b H_c(x-t)\,dH_d(x-t) + \int_a^b H_d(x-t) \, dH_c(x-t) = c + d;
\]
On the other hand, $H_d(b)H_c(b) - H_d(a)H_c(a)=1$.  Consequently, we need to introduce a term, $A(t)$, into the formula for integration by parts  that compensates for a common discontinuity at a point $t$.  This term is the same as the the error term introduced in~\cite[Theorem 6.2.2]{LebesgueStieltjesCarter} for the Lebesgue-Stieltjes integral. 

\begin{theorem}\label{thm:IntByParts}
    Fix $\alpha,\beta\in BV[a,b]$. If  $S$ is the set of common discontinuities of $\alpha$ and $\beta$, then 
    $$\int_a^b \alpha(x)\, d \beta + \int_a^b \beta(x)\, d \alpha = \alpha(b)\beta(b)-\alpha(a)\beta(a) + \sum_{t\in S} A_{\alpha,\beta}(t),$$ 
    where 
    $$A_{\alpha,\beta}(t)= [\alpha(t)-\tfrac{1}{2}(\alpha(t+)+\alpha(t-))]\mu_\beta(\{t\})+[\beta(t)-\tfrac{1}{2}(\beta(t+)+\beta(t-))]\mu_\alpha(\{t\}).$$
\end{theorem}

\begin{remark}
Arguing as in~\cite[Lemma~5.2]{TSI}, we see that $A(t)=0$ if the discontinuities  of $\alpha$ and $\beta$ at $t$ are balanced; that is, 
$$\alpha(t)=\frac{\alpha(t-)+\alpha(t+)}{2}, \qquad \beta(t) = \frac{\beta(t-)+\beta(t+)}{2}.$$ 
\end{remark}

\medskip

The proof of Theorem~\ref{thm:IntByParts} requires three lemmas whose proofs we defer until after the main proof.

\begin{lemma}\label{lem:saltuscorrectionterm}
    Given $\alpha,\beta\in BV[a,b]$, with reduced saltus decompositions
    $$\alpha = G\alpha + S_L \alpha + S_R \alpha,\quad \beta = G\beta + S_L\beta +S_R\beta,$$ 
    we have that
    $$A_{\alpha, \beta}(t) = A_{S_L\alpha, S_L\beta}(t)+A_{S_R\alpha, S_R\beta}(t).$$
\end{lemma}

\begin{lemma}\label{lem:IBPcross}
    Given $\alpha,\beta\in BV[a,b]$, if $\alpha$ is continuous and $\beta$ is a reduced, left or right continuous saltus function, then $$\int_a^b \alpha(x)\, d \beta + \int_a^b \beta(x)\, d \alpha = \alpha(b)\beta(b)-\alpha(a)\beta(a).$$
\end{lemma}

\begin{lemma}\label{lem:IBPsaltus}
    Given $\alpha,\beta\in BV[a,b]$, suppose they are both reduced,  left or right continuous saltus functions.  If $S$ is the set of common discontinuities of $\alpha$ and $\beta$, then 
    $$\int_a^b \alpha(x)\, d \beta + \int_a^b \beta(x)\, d \alpha 
    = \alpha(b)\beta(b)-\alpha(a)\beta(a) + \sum_{t\in S} A_{\alpha,\beta}(t).$$
\end{lemma}

\begin{proof}[Proof of Theorem \ref{thm:IntByParts}]
We first consider three special cases: (i) both $\alpha, \beta$ are continuous; (ii) $\alpha$ is continuous, $\beta$ is a reduced, left or right continuous saltus function;  (iii) both $\alpha,\beta$ are reduced, left or right continuous saltus functions. In Case (i), since both $\alpha$ and $\beta$ are continuous, by Theorems~\ref{thm:RDSiffDSintegrator} and~\ref{thm:RDSiffDSintegrand} the RDS and DS-integrals agree, and  the desired identity follows from~\cite[Proposition 5.5]{TSI}.  Case (ii) follows from~Lemma~\ref{lem:IBPcross}. Case (iii) follows from Lemma~\ref{lem:IBPsaltus}.

We now argue in general. Fix $\alpha,\,\beta \in BV[a,b]$.
    By  Proposition \ref{prop:ReducedSaltusfunctions}, we have the saltus decompositions 
    $$\alpha = G\alpha + S_L \alpha + S_R \alpha,\quad \beta = G\beta + S_L\beta +S_R\beta,$$ 
    where the saltus functions are reduced. Then by Theorems~\ref{thm:linearity} and~\ref{thm:linintegrator}, and by Lemma~\ref{lem:saltuscorrectionterm} we can decompose the integrals and argue as follows:
  \begin{align*}
        \int_a^b \alpha(x)\, d \beta +\int_a^b \beta(x)\, d \alpha &= \int_a^b G\alpha(x)+S_L\alpha(x)+S_R\alpha(x)\, d [G\beta+S_L\beta+S_R\beta]\\
        &\quad +\int_a^b G\beta(x)+S_L\beta(x)+S_R\beta(x)\, d [G\alpha+S_L\alpha+S_R\alpha]\\
        &= \alpha(b)\beta(b)-\alpha(a)\beta(a)+ \sum_{t\in S}A_{S_L\alpha, S_L\beta}(t)+\sum_{t\in S}A_{S_R\alpha, S_R\beta}(t)\\
        &=\alpha(b)\beta(b)-\alpha(a)\beta(a)+ \sum_{t\in S}A_{\alpha, \beta}(t),
    \end{align*}
where to get the term $\alpha(b)\beta(b)-\alpha(a)\beta(a)$ we have used the saltus decomposition to recombine terms gotten from Cases (i) and (ii).
\end{proof}

\begin{proof}[Proof of Lemma \ref{lem:saltuscorrectionterm}]
    Observe that 
    \begin{align*}
        \mu_{\alpha}(\{t\})= \alpha(t+)-\alpha(t-)&= G\alpha(t)+S_L\alpha(t+)+S_R\alpha(t)\\
        &\quad -G\alpha(t)-S_L\alpha(t)-S_R\alpha(t-)\\ &= \mu_{S_L\alpha}(\{t\})+\mu_{S_R\alpha}(\{t\}).
    \end{align*}
    Hence, \begin{align*}
        A_{\alpha,\beta}(t)&= [\alpha(t)-\tfrac{1}{2}(\alpha(t+)+\alpha(t-))]\mu_\beta(\{t\})+[\beta(t)-\tfrac{1}{2}(\beta(t+)+\beta(t-))]\mu_\alpha(\{t\})\\
        &= \tfrac{1}{2}[S_L\alpha(t)+S_R\alpha(t)-S_L\alpha(t+)-S_R\alpha(t-)][\mu_{S_L\beta}(\{t\})+\mu_{S_R\beta}(\{t\})]\\
        &\quad + \tfrac{1}{2}[S_L\beta(t)+S_R\beta(t)-S_L\beta(t+)-S_R\beta(t-)][\mu_{S_L\alpha}(\{t\})+\mu_{S_R\alpha}(\{t\})]\\
        &= \tfrac{1}{2}[\mu_{S_R\alpha}(\{t\})-\mu_{S_L\alpha}(\{t\})][\mu_{S_L\beta}(\{t\})+\mu_{S_R\beta}(\{t\})]\\
        &\quad + \tfrac{1}{2}[\mu_{S_R\beta}(\{t\})-\mu_{S_L\beta}(\{t\})][\mu_{S_L\alpha}(\{t\})+\mu_{S_R\alpha}(\{t\})]\\
        &= \mu_{S_R\alpha}(\{t\})\mu_{S_R\beta}(\{t\})- \mu_{S_L\alpha}(\{t\})\mu_{S_L\beta}(\{t\})\\
        &= A_{S_R\alpha,S_R\beta}(t)+A_{S_L\alpha,S_L\beta}(t).
    \end{align*}
\end{proof}

\begin{proof}[Proof of Lemma \ref{lem:IBPcross}]
    Let $\{(x_i, b_i)\}_{i=1}^\infty$ be the saltus set associated to $\beta$;
    then
    $$\beta(x)= \sum_{i=1}^\infty b_i H_c(x-x_i),$$ 
    where $c\in\{0, 1\}$. By Theorem \ref{thm:BddinRDSsaltus}, $$\int_a^b \alpha(x)\, d \beta = \sum_{i=1}^\infty b_i \alpha(x_i).$$

    Let $\beta_n(x) = \sum_{i=1}^n b_iH_c(x-x_i)$. By \cite[Lemma 3.36]{TSI}, $\beta_n\to \beta$ uniformly. By Theorem~\ref{thm:uniformconv}, 
    \begin{multline*}
        \int_a^b \beta(x)\, d \alpha
        = \lim_{n\to \infty}\int_a^b \beta_n(x)\, d \alpha
        = \lim_{n\to \infty}\sum_{i=1}^nb_i\int_a^b H_c(x-x_i) \, d \alpha\\
        =\lim_{n\to \infty}\sum_{i=1}^nb_i(\alpha(b)-\alpha(x_i))
        = \alpha(b)\beta(b)- \sum_{i=1}^\infty b_i\alpha(x_i).
    \end{multline*}

    By Proposition \ref{prop:ReducedSaltusfunctions}, $\beta(a)=0$. Hence, $$\int_a^b \alpha(x)\, d \beta + \int_a^b \beta(x)\, d \alpha = \alpha(b)\beta(b)= \alpha(b)\beta(b)-\alpha(a)\beta(a).$$
\end{proof}

\begin{proof}[Proof of Lemma \ref{lem:IBPsaltus}]
    Let $$\alpha(x) = \sum_{i=1}^\infty a_i H_{c}(x-x_i), \quad \beta(x) = \sum_{j=1}^\infty b_j H_{d}(x-t_j).$$ 
    First, suppose $c=d=0$. Then by left continuity, 
    $$A(t)= -(\alpha(t+)-\alpha(t-))(\beta(t+)-\beta(t-)).$$ 
    By Theorem \ref{thm:BddinRDSsaltus}, we have 
    \begin{align*}
        \int_a^b \alpha(x)\, d \beta + \int_a^b \beta(x)\, d \alpha &= \sum_{j=1}^\infty b_j\alpha(t_j)+\sum_{i=1}^\infty a_i\beta(x_i)\\
        &= \sum_{j=1}^\infty b_j\sum_{i\; :\; x_i<t_j} a_i+\sum_{i=1}^\infty a_i\sum_{j\; :\; t_j<x_i } b_j\\
        &= \sum_{i=1}^\infty \sum_{j\; :\; t_j>x_i}a_ib_j+\sum_{i=1}^\infty \sum_{j\; :\; t_j<x_i } a_ib_j\\
        &= \sum_{i=1}^\infty \sum_{j\; :\; t_j\neq x_i} a_ib_j.
    \end{align*}

    Now fix $i\in \mathbb{N}$. If $x_i\not\in \{t_j\}_{j=1}^\infty$, then 
    \begin{align*}
        \sum_{j\; :\; t_j\neq x_i} a_ib_j &= a_i\sum_{j=1}^\infty b_j= a_i\beta(b).
    \end{align*}
    Next, suppose $x_i \in \{t_j\}_{j=1}^\infty$. That is, there exists $j_0$ such that $x_i=t_{j_0}$. Then, \begin{align*}
        \sum_{j\; :\; t_j\neq x_i} a_ib_j &= a_i\sum_{j=1}^\infty b_j - a_ib_{j_0}= a_i\beta(b)-a_ib_{j_0}.
    \end{align*}
    Hence, \begin{align*}
        \int_a^b \alpha(x)\, d \beta + \int_a^b \beta(x)\, d \alpha &= \alpha(b)\beta(b)-\sum_{t\in S} (\alpha(t+)-\alpha(t-))(\beta(t+)-\beta(t-))\\
        &= \alpha(b)\beta(b)- \alpha(a)\beta(a)+\sum_{t\in S} A_{\alpha,\beta}(t)\\
    \end{align*}

    The case $c=d=1$ follows similarly as the previous case. 
    
    Finally, suppose $c=0, d=1$. Observe that $A_{\alpha,\beta}(t)=0$. Then, by Theorem \ref{thm:BddinRDSsaltus}, we have \begin{align*}
        \int_a^b \alpha(x)\, d \beta + \int_a^b \beta(x)\, d \alpha &= \sum_{j=1}^\infty b_j\alpha(t_j)+\sum_{i=1}^\infty a_i\beta(x_i)\\
        &= \sum_{j=1}^\infty b_j\sum_{i\; :\; x_i<t_j} a_i+\sum_{i=1}^\infty a_i\sum_{j\; :\; t_j\le x_i } b_j\\
        &= \sum_{i=1}^\infty \sum_{j\; :\; t_j>x_i}a_ib_j+\sum_{i=1}^\infty \sum_{j\; :\; t_j\le x_i } a_ib_j\\
        &= \sum_{i=1}^\infty \sum_{j=1}^\infty a_ib_j\\
        &= \alpha(b)\beta(b)-\alpha(a)\beta(b).
    \end{align*}
\end{proof}

\section{The Bounded Convergence Theorem}
\label{section:BCT}

In this section, we prove the Bounded Convergence Theorem for the RDS integral. As we noted before, this theorem is not true for the interior DS-integral defined in~\cite{TSI} unless $\alpha$ is continuous; our new definition therefore overcomes this technical problem.  For the convenience of the reader we restate it from Section~\ref{section:intro}.

\begin{theorem}\label{thm:BddConvergence}
Given $\alpha \in BV[a,b]$ and a uniformly bounded sequence of  functions $\{f_n\}_{n=1}^\infty$ on $[a,b]$, if 

\begin{itemize}
    \item for all $n\in \mathbb{N}$, $f_n\in RDS_\alpha[a,b]$,
    
    \item $f_n\to f$ pointwise as $n\to \infty$,
    
    \item $f\in RDS_\alpha[a,b]$,
\end{itemize}

then
\begin{equation} \label{eqn:BddConvergence1}
\lim_{n\to\infty}\int_a^b f_n(x)\, d \alpha = \int_a^b f(x)\, d \alpha.
\end{equation}
\end{theorem}

Our proof of Theorem~\ref{thm:BddConvergence} is modeled on the proof  of the bounded convergence theorem for the Riemann integral due to Luxemburg~\cite{BCT}. We have organized our proof as follows.  In the main part of the proof, we show that  we can reduce to the case when our sequence of functions $\{f_n\}_{n=1}^\infty$ is nonnegative and decreasing, and the integrator $\alpha$ is continuous; this special case is given in Lemma~\ref{lem:LowerIntConvergence}.  To prove this lemma we need to approximate this decreasing sequence by a sequence of continuous functions; this is done in Lemma~\ref{lem:ContApprox}.  Given this lemma, to prove  Lemma~\ref{lem:LowerIntConvergence} we  apply Dini's Theorem to get uniform convergence, and finally use Theorem~\ref{thm:uniformconv}.

\begin{lemma}\label{lem:LowerIntConvergence}
Given $\alpha \in I[a,b]$ continuous and a sequence of functions $\{f_n\}_{n=1}^\infty$ on $[a,b]$, if 

\begin{itemize}
    \item $\{f_n\}_{n=1}^\infty$ is pointwise decreasing,
    
    \item $f_n(x)\ge 0$ for every $x\in [a,b]$ and $n\in \mathbb{N}$,
    
    \item $f_n\to 0$ pointwise as $n\to \infty$,

\end{itemize}

then $$\lim_{n\to\infty}L_\alpha(f_n, [a,b]) = 0.$$
\end{lemma}

We defer the proof of Lemma \ref{lem:LowerIntConvergence} to below.

\begin{proof}[Proof of Theorem \ref{thm:BddConvergence}]
Suppose first that $\alpha\in I[a,b]$. We will first prove the theorem for the special case when $f_n(x)\ge 0$ for all $x\in[a,b],$ $n\in \mathbb{N}$, and $f_n\to 0$ pointwise. By Proposition~\ref{prop:ReducedSaltusfunctions}, we can decompose 
$$\alpha = G\alpha + S_L\alpha+S_R\alpha $$ 
where $G\alpha$ is continuous and $S_L\alpha$ and $S_R\alpha$ are reduced saltus functions.
By Theorem \ref{thm:BddinRDSsaltus} and Lemma \ref{lem:RDSGalpha}, we have $f_n\in RDS_{G\alpha}[a,b]$, $f_n\in RDS_{S_L\alpha}[a,b]$ and $f_n\in RDS_{S_R\alpha}[a,b]$ for every $n\in \mathbb{N}$.  By Theorem~\ref{thm:linintegrator},
\begin{align*}
    \lim_{n\to \infty} \int_a^b f_n(x)\, d \alpha&= \lim_{n\to \infty}\left[\int_a^b f_n(x)\, d G\alpha+\int_a^b f_n(x)\, d S_L\alpha+\int_a^b f_n(x)\, d S_R\alpha\right],
\end{align*}
and so to prove \eqref{eqn:BddConvergence1} it will suffice to show that each of the integrals on the righthand side converges to $0$ and $n \rightarrow \infty$.

We first consider $G\alpha$.  For each $n\in \mathbb{N}$, define $p_n= \sup_{k\ge n}f_k.$ Then, $0\le f_n(x)\le p_n(x)$ for $x\in [a,b]$ and $\{p_n\}_{n=1}^\infty$ decreases pointwise. In particular, for $x\in [a,b]$, $$\lim_{n\to \infty}p_n(x) = \limsup_{n\to \infty} f_n(x) = \lim_{n\to \infty}f_n(x)=0.$$ 
Hence, by Lemma \ref{lem:LowerIntConvergence},
\begin{multline*}
    0
    \le\liminf_{n\to \infty}\int_a^b f_n(x)\, d G\alpha 
    \le \limsup_{n\to \infty}\int_a^b f_n(x)\, d G\alpha \\
    = \limsup_{n\to\infty}L_{G\alpha}(f_n, [a,b])
   \le \limsup_{n\to\infty}L_{G\alpha}(p_n, [a,b])
    =\lim_{n\to\infty}L_{G\alpha}(p_n, [a,b])
    =0.
\end{multline*}
Therefore,
$$\lim_{n\to \infty} \int_a^bf_n(x)\, d G\alpha = 0.$$

We next consider $S_L\alpha$.  Let $\{(x_i, a_i)\}$ be the saltus set for $S_L\alpha$. Then we can write
$S_L\alpha = \sum_{i=1}^\infty a_iH_c(x-x_i)$ where $c=0$.
By Lemma \ref{thm:BddinRDSsaltus}, for each $n\in \mathbb{N}$, $$\int_a^b f_n(x)\, d S_L\alpha = \sum_{i=1}^\infty a_if_n(x_i).$$
Let $M$ be the uniform bound for $\{f_n\}_{n=1}^\infty$.  By Definition \ref{def:saltusset}, 
$$\sum_{i=1}^\infty |a_i||f_n(x_i)|\le \sum_{i=1}^\infty M|a_i|<\infty.$$ 
Thus,  for $n\in \mathbb{N}$, 
$$\sum_{i=1}^\infty a_if_n(x_i)$$ converges absolutely.

Fix $\epsilon>0$. Let $K\in\mathbb{N}$ be such that $$\sum_{i=K+1}^\infty M|a_i|<\frac{\epsilon}{2}.$$ For each $1\le i\le K$, let $N_i\in \mathbb{N}$ be such that for $n\ge N_i$, $f_{n}(x_i)<\frac{\epsilon}{2|a_i|K}.$  If we set $N=\max_{1\le i\le K}\{N_i\}$, then  for $n\ge N$,
\begin{multline*}
    \left|\sum_{i=1}^\infty a_if_n(x_i)\right| 
    \le  \left|\sum_{i=1}^K a_if_n(x_i)\right|+\left|\sum_{i=K+1}^\infty a_if_n(x_i)\right|\\
    \le \sum_{i=1}^{K}|a_i|\frac{\epsilon}{2|a_i|K}+\sum_{i=K+1}^\infty M|a_i|
    < \frac{\epsilon}{2}+\frac{\epsilon}{2}
    =\epsilon.
\end{multline*}
Hence, $$\lim_{n\to \infty}\int_a^bf_n(x)\, d S_L\alpha =0.$$ 

Finally, for $S_R\alpha$, we can apply the same argument with $c=1$ to get 
$$\lim_{n\to \infty}\int_a^bf_n(x)\, d S_R\alpha =0.$$
This completes the proof of this special case.

\medskip

Now we will prove the theorem for the general case of $\{f_n\}_{n=1}^\infty$. Define for each $n\in\mathbb{N},$ $$g_n(x)= |f_n(x)-f(x)|.$$ Note that $g_n \in RDS_{\alpha}[a,b]$ and $\{g_n\}_{n=1}^\infty$ is a uniformly bounded sequence of functions. By the previous case, 
$$\lim_{n\to \infty}\int_a^b |f_n(x)-f(x)|\, d \alpha = \lim_{n\to \infty}\int_a^b g_n(x)\, d \alpha  = 0.$$ 
By Lemma \ref{lem:triangleineq}, 
$$0\le \left|\int_a^b f_n(x)-f(x)\, d \alpha\right|\le \int_a^b |f_n(x)-f(x)|\, d \alpha,$$
and so,  by the squeeze theorem, 
$$\lim_{n\to\infty}\int_a^b f_n(x)- f(x)\, d \alpha=0,$$ 
which yields \eqref{eqn:BddConvergence1}.

\medskip

Finally, we prove the theorem for  $\alpha \in BV[a,b]$. By Theorem \ref{thm:pnalpha}, $f_n\in RDS_{P\alpha}[a,b]$ for every $n\in \mathbb{N}$ and $f\in RDS_{P\alpha}[a,b]$. Therefore, by the previous case applied to $P\alpha$ and $N\alpha$,
\begin{multline*}
    \lim_{n\to \infty}\int_a^b f_n(x)\, d \alpha 
    = \lim_{n\to \infty}\bigg(\int_a^b f_n(x)\, d P\alpha-\int_a^b f_n(x)\, d N\alpha\bigg)\\
    = \int_a^b f(x)\, d P\alpha-\int_a^b f(x)\, d N\alpha
    =\int_a^b f(x)\, d \alpha.
\end{multline*}
\end{proof}

To prove Lemma \ref{lem:LowerIntConvergence} we need another lemma, whose proof we give below.

\begin{lemma}\label{lem:ContApprox}
Given $\alpha \in I[a,b]$ continuous and $f\in B[a,b]$ nonnegative, for each $\epsilon>0$, there exists $g\in C[a,b]$ such that $0\le g(x)\le f(x)$ for every $x\in [a,b]$ and $$L_\alpha(f, [a,b]) < \int_a^b g(x)\, d \alpha +\epsilon.$$ 
\end{lemma}

\begin{proof}[Proof of Lemma \ref{lem:LowerIntConvergence}]
Fix $\epsilon>0$. By Lemma \ref{lem:ContApprox}, for each $n\in \mathbb{N}$, there exists a continuous function $g_n$ such that $0\le g_n(x)\le f_n(x)$ for $x\in [a,b]$ and  
\begin{equation} \label{eqn:lemma83bound}
L_\alpha(f_n, [a,b]) < \int_a^b g_n(x)\, d \alpha +\frac{\epsilon}{2^n}.
\end{equation}
For each $n\in \mathbb{N}$, define $h_n=\min\{g_1, \dots, g_n\}$. Then, $0\le h_n\le g_n\le f_n$ and $h_n\in C[a,b]$. Thus, $\{h_n\}_{n=1}^{\infty}$ decreases to $0$ pointwise. Therefore, by Dini's Theorem \cite[Exercise~24.H]{MR0393369}, the sequence $\{h_n\}_{n=1}^\infty$ converges uniformly to $0$ on $[a,b]$. Hence, by Theorem \ref{thm:uniformconv}, $$\lim_{n\to \infty}\int_a^b h_n(x)\, d \alpha=0.$$

Fix $n\in \mathbb{N}$ and $x\in [a,b]$. Then, $h_n(x)=g_j(x)$ for some $1\le j\le n$. Thus, 
\begin{align}
    & h_n(x) + \sum_{i=1}^n\big(\max\{g_i(x), \dots, g_n(x)\}-g_i(x)\big) \label{eqn1} \\
    & \qquad \qquad = g_j(x) + \sum_{i=1}^n(\max\{g_i(x), \dots, g_n(x)\}-g_i(x)) \notag\\
    &\qquad \qquad \ge  g_j(x)+ \max\{g_j(x), \dots, g_n(x)\} -g_j(x) \notag\\\
    & \qquad \qquad = \max\{g_j(x), \dots, g_n(x)\}\notag\\\
    &\qquad \qquad \ge g_n(x).\notag\
\end{align}
Since $x$ is arbitrary, this is true for all $x\in [a,b]$.

For $1\le i\le n$, $\max\{g_i, \dots, g_n\}$ is continuous, so by Theorem \ref{thm:continuousinRDS}, 
$\max\{g_i, \dots, g_n\}\in RDS_{\alpha}[a,b]$. Moreover, since $\max\{g_i, \dots, g_n\}\le \max\{f_i, \dots, f_n\}=f_i$ as $\{f_i\}_{n=1}^\infty$ is pointwise decreasing, for $1\le i\le n$ we have that
\begin{align*}
    L_\alpha(f_i, [a,b]) &\ge L_\alpha(\max\{g_i, \dots, g_n\}, [a,b]) \\
    &= \int_a^b \max\{g_i(x), \dots, g_n(x)\}\, d \alpha\\
    &= \int_a^b \max\{g_i(x), \dots, g_n(x)\}-g_i(x)\, d \alpha+\int_a^b g_i(x)\, d \alpha.
\end{align*}
If we combine this estimate with~\eqref{eqn:lemma83bound}, we get
\begin{equation} \label{eqn2}
    \int_a^b \max\{g_i(x), \dots, g_n(x)\}-g_i(x)\, d \alpha 
    \le L_\alpha(f_i, [a,b]) - \int_a^b g_i(x)\, d \alpha <\frac{\epsilon}{2^i}.
\end{equation}

Therefore, by Theorem \ref{thm:monotonicity} and by \eqref{eqn1} and \eqref{eqn2}, \begin{align*}
    0\le L_\alpha(f_n, [a,b]) &< \frac{\epsilon}{2^n}+ \int_a^b g_n(x)\, d \alpha\\
    &\le \frac{\epsilon}{2^n}+ \int_a^b h_n(x)\, d \alpha+ \sum_{i=1}^n \int_a^b \max\{g_i(x), \dots, g_n(x)\}-g_i(x)\, d \alpha\\
    &\le \frac{\epsilon}{2^n}+ \int_a^b h_n(x)\, d \alpha+ \sum_{i=1}^n \frac{\epsilon}{2^i}\\
    &= \epsilon+\int_a^b h_n(x)\, d \alpha.
\end{align*}

Thus, for every $\epsilon>0$, $$0\le\liminf_{n\to\infty}L_\alpha(f_n, [a,b])\le  \limsup_{n\to\infty}L_\alpha(f_n, [a,b])\le \epsilon+ \lim_{n\to\infty}\int_a^b h_n(x)\, d \alpha = \epsilon, $$ which gives $$\limsup_{n\to\infty}L_\alpha(f_n, [a,b])=\limsup_{n\to\infty}L_\alpha(f_n, [a,b])=0.$$ Hence, $$\lim_{n\to \infty} L_\alpha(f_n, [a,b]) = 0.$$
\end{proof}

\begin{proof}[Proof of Lemma \ref{lem:ContApprox}]
Fix $\epsilon>0$. By Definition \ref{def:RDSint}, there exists a step function $v(x)\le f(x)$ defined with respect to the partition $\{x_i\}_{i=1}^N$ such that $v(x)=c_i$ for $x\in I_i$ and 
$$L_\alpha(f, [a,b]) \le \int_a^b v(x)\, d \alpha +\frac{\epsilon}{2}.$$ 
Without loss of generality, by Corollary \ref{cor:RDScritbdd} we may assume that  $v(x)\ge 0$ for $x\in [a,b]$. Let $M=\max_{1\le i\le N}\{c_i\}+1$. Since $\alpha$ is continuous on $[a,b]$, it is uniformly continuous on $[a,b]$. Thus, there exists a $\delta_1>0$ such that for every $x,y\in [a,b]$ with $|x-y|<\delta_1$, $$|\alpha(x)-\alpha(y)|<\frac{\epsilon}{4MN}.$$ Let $\delta = \min(\frac{1}{2}\delta_1, \min_{1\le i\le N}\{\frac{x_i-x_{i-1}}{4}\}).$ 

We now define the function $g$ on $[a,b]$ by redefining $f$ in each neighborhood $(x_i-\delta,x_i+\delta)$, $1\leq i \leq n-1$, if $x_i$ is a discontinuity of $f$.  (If $x_0$ or $x_n$ is a discontinuity, this argument immediately adapts to the intervals $[x_0,x_0+\delta)$, $(x_n-\delta,x_n]$.)   For each $i$, if $v(x_i)\le \min\{c_i, c_{i+1}\}$, let $g$ be the linear function such that $g(x_i-\delta)=c_i, g(x_i)=v(x_i), g(x_i+\delta)=c_{i+1}$.  On the other hand,  if $v(x_i)> \min\{c_i, c_{i+1}\}$, let $g$ be the piecewise linear function with $g(x_i-\delta)=c_i, g(x_i)=\min\{c_i, c_{i+1}\}, g(x_i+\delta)=c_{i+1}$. Since the pieces of $g$ are continuous and agree at the endpoints, $g$ is continuous. Moreover, 
$$0\le g(x)\le v(x)\le f(x).$$ 

We now estimate as follows: since $\alpha$ is continuous and increasing, by the definition of the RDS-integral of a step function,

\begin{align*}
    \int_a^b g(x)\, d \alpha 
    & = \sum_{i=1}^N \int_{x_{i-1}}^{x_i} g(x)\, d \alpha \\
    &\ge  \sum_{i=1}^N \int_{x_{i-1}+\delta}^{x_i-\delta} g(x)\, d \alpha\\
    &= \sum_{i=1}^N c_i\mu_{\alpha}((x_{i-1}+\delta, x_i-\delta))\\
    &= \sum_{i=1}^N c_i\mu_{\alpha}((x_{i-1}, x_i)) - \sum_{i=1}^N c_i[\mu_{\alpha}((x_{i-1}, x_{i-1}+\delta])+\mu_\alpha([x_i-\delta, x_i))]\\
    &=\int_a^b v(x)\, d \alpha - \sum_{i=1}^N c_i[|\alpha(x_{i-1}+\delta)-\alpha(x_{i-1})|+|\alpha(x_i)-\alpha(x_i-\delta)|]\\
    &> \int_a^b v(x)\, d \alpha - \sum_{i=1}^N c_i\frac{\epsilon}{2NM}\\
    &\ge \int_a^b v(x)\, d \alpha - \sum_{i=1}^N M\frac{\epsilon}{2NM}\\
    &= \int_a^b v(x)\, d \alpha - \frac{\epsilon}{2}.\\
\end{align*}
Therefore, we have 
$$L_\alpha(f, [a,b])\le \int_a^b v(x)\, d \alpha +\frac{\epsilon}{2}< \int_a^b g(x)\, d \alpha +\epsilon.$$
\end{proof}

\section{RDS Integral and Related Integrals}
\label{section:alt-defns}

In this section we continue our discussion of alternative  definitions of  the Stieltjes integral, begun in Section~\ref{section:intro}, and compare them to the  RDS-integral. As we showed above in Corollary~\ref{cor:bestfitRDScrit}, our definition of the RDS-integral is equivalent to the original definition by Ross, Definition~\ref{def: RossIntegral}.  Further, as shown in~\cite{TSI}, the interior DS-integral is more general than the Riemann-Stieltjes and exterior Darboux-Stieltjes integrals, and agrees with them when they exist.  Therefore, our first step is to characterize when the RDS-integral and the interior DS-integral agree.  We began this in Theorems~\ref{thm:RDSiffDSintegrator} and~\ref{thm:RDSiffDSintegrand}, but here we will use Theorem~\ref{thm:RDSLebcrit} to give a complete characterization.  Then we will  define Riemann and Pollard-type integrals corresponding to the RDS-integral. The Pollard-type integral  is equivalent to the RDS-integral while the Riemann-type integral is still weaker. We will also briefly consider a variant of the RS-integral introduced by Ross that is equivalent to the RDS-integral.   Finally, we show the RDS-integral agrees with the Lebesgue-Stieltjes integral whenever the former is defined.

\medskip

To compare the RDS-integral and the interior DS-integral, we will use the following characterization of interior DS-integrability proved in \cite[Theorem~5.19]{TSI}.

\begin{theorem} \label{thm:ds-lebesgue-crit}
Given $\alpha \in I[a,b]$ and $f\in B[a,b]$, $f \in DS_\alpha[a,b]$ if and only if:
\begin{enumerate}
    \item the subset of $[a,b]$ where $f$ is discontinuous has $G\alpha$-measure $0$, where $G\alpha$ is the continous part of the saltus decomposition of $\alpha$;

    \item If $\{x_i\}_{i=1}^\infty$ is the set of discontinuities of $\alpha$, then for each $i$, 
    \begin{enumerate}
        \item $\alpha$ is right continuous at $x_i$ or $f(x_i+)$ exists, and

        \item $\alpha$ is left continuous at $x_i$ or $f(x_i-)$ exists.
    \end{enumerate}
\end{enumerate}
\end{theorem}

As an immediate consequence of Theorems~\ref{thm:RDSLebcrit} and~\ref{thm:ds-lebesgue-crit}, we have the following generalization of Theorems~\ref{thm:RDSiffDSintegrator} and~\ref{thm:RDSiffDSintegrand}.  Note that the assumption that $f$ and $\alpha$ have no common discontinuities is necessary for the RS-integral to exist and be equal to the DS-integral, and so this is a natural sufficient condition to consider.

\begin{theorem}\label{thm:RDSiffDSsuffcond}
    Given $\alpha\in BV[a,b]$ and $f\in B[a,b]$, suppose $f$ and $\alpha$ have no common discontinuities. Then $f\in RDS_\alpha[a,b]$ if and only if $f\in DS_\alpha[a,b]$. Moreover,
    $$\int_a^b f(x)\, d\alpha = (DS)\int_a^b f(x)\, d\alpha.$$
\end{theorem}

\begin{proof}
Since $f$ and $\alpha$ have no common discontinuities, $f(x\pm)=f(x)$ at each discontinuity $x$ of $\alpha$, so the second condition in Theorem~\ref{thm:ds-lebesgue-crit} is automatically fulfilled.  By Remark~\ref{remark:two-defns}, since $G\alpha$ is continuous, the definition of $G\alpha$-measure $0$ used in Theorem \ref{thm:RDSLebcrit} agrees with that used in Theorem~\ref{thm:RDSLebcrit}.  Hence, $f\in DS_\alpha[a,b]$ if and only if $f\in RDS_\alpha[a,b]$. 
    
  We now prove that the two integrals agree.   By Proposition~\ref{prop:ReducedSaltusfunctions} we have the reduced saltus decomposition
  $\alpha=G\alpha+S_L\alpha +S_R\alpha$.   Let $\{(x_i, a_i)\}$ be the reduced saltus set for $S_L\alpha$ and $\{(y_i, b_i)\}$ be the reduced saltus set for $S_R\alpha$. By Theorems~\ref{thm:linintegrator}, \ref{thm:RDSiffDSintegrator}, and~\ref{thm:BddinRDSsaltus}  we have
  \begin{align*}
        \int_a^b f(x)\, d\alpha &= \int_a^b f(x)\, dG\alpha+\int_a^b f(x)\, dS_R\alpha+\int_a^b f(x)\, dS_L\alpha\\
        &= (DS)\int_a^b f(x)\, dG\alpha + \sum_{i=1}^\infty a_if(x_i) + \sum_{i=1}^\infty b_if(y_i)\\
        &= (DS)\int_a^b f(x)\, dG\alpha + \sum_{i=1}^\infty a_if(x_i+) + \sum_{i=1}^\infty b_if(y_i-);\\
        \intertext{by the corresponding results for the DS-integral, \cite[Proposition 5.9]{TSI} and \cite[Theorem 4.30c]{TSI},}
        &= (DS)\int_a^b f(x)\, dG\alpha+(DS)\int_a^b f(x)\, dS_R\alpha+(DS)\int_a^b f(x)\, dS_L\alpha\\
        &= (DS)\int_a^b f(x)\, d\alpha.
    \end{align*}
\end{proof} 

\medskip

Arguing as we did in the proof of Theorem~\ref{thm:RDSiffDSsuffcond}, it is clear from Theorems~\ref{thm:RDSLebcrit} and~\ref{thm:ds-lebesgue-crit} that if a function $f$ is interior DS-integrable, it is RDS-integrable.  However, as we showed in Lemma~\ref{lem:H_c}, the two integrals may not agree if they share a common discontinuity.  By carefully considering the additional continuity conditions in the second part of Theorem~\ref{thm:ds-lebesgue-crit}, we can give a complete characterization of when the two integrals agree in terms of a condition reminiscent of the error term introduced for integration by parts, Theorem~\ref{thm:IntByParts}.

\begin{theorem}\label{thm: RDSiffDScharacterization}
    Given $\alpha\in I[a,b]$, let $G\alpha+S_L\alpha+S_R\alpha$ be its reduced saltus decomposition, with $\{(x_i, a_i)\}$ and $\{(y_i, b_i)\}$ the reduced saltus sets for $S_L\alpha$ and $S_R\alpha$.
    Then for any $f\in B[a,b]$,  $f\in DS_\alpha[a,b]$ and 
    \begin{equation} \label{eqn:RDSiffDS-1}
    \int_a^b f(x)\, d\alpha = (DS)\int_a^b f(x)\, d\alpha
    \end{equation}
    if and only if $f\in RDS_\alpha[a,b]$, $f(x_i+)$ and $f(y_i-)$ exist for all $i$ and  
    \begin{equation} \label{eqn:RDSiffDS-2}
\sum_{i=1}^\infty a_i[f(x_i)-f(x_i+)]+\sum_{i=1}^\infty b_i[f(y_i)-f(y_i-)]=0.
\end{equation}
\end{theorem}

\begin{proof}
  First suppose $f\in DS_\alpha[a,b]$ (and so $f\in RDS_\alpha[a,b]$) and 
  $$\int_a^b f(x)\, d\alpha = (DS)\int_a^b f(x)\, d\alpha.$$ 
  We will prove that \eqref{eqn:RDSiffDS-2} holds.  By \cite[Proposition 4.22]{TSI}, $f\in DS_{S_L\alpha} [a,b]$ and $f\in DS_{S_R\alpha} [a,b]$. Then, by Theorem~\ref{thm:ds-lebesgue-crit}, $S_L\alpha$ is not right continuous for any $x_i$ and so $f(x_i+)$ exists. Similarly, $S_R\alpha$ is not left continuous for any $y_i$ and so $f(y_i-)$ exists. Further, if we apply  Theorem \ref{thm:linintegrator} and \cite[Theorem 4.30c]{TSI} to each side of \eqref{eqn:RDSiffDS-1}, we get that 
  \begin{multline*}
       \int_a^b f(x)\, dG\alpha+\int_a^b f(x)\, dS_L\alpha+\int_a^b f(x)\, dS_R\alpha\\ =(DS)\int_a^b f(x)\, dG\alpha+(DS)\int_a^b f(x)\, dS_L\alpha+(DS)\int_a^b f(x)\, dS_R\alpha.
   \end{multline*} 
   Moreover, by Theorem \ref{thm:RDSiffDSintegrator}, we have 
   $$\int_a^b f(x)\, dG\alpha=(DS)\int_a^b f(x)\, dG\alpha.$$ 
   Therefore, by Theorem \ref{thm:BddinRDSsaltus} and \cite[Proposition 5.9]{TSI}, 
   \begin{align*}
       \sum_{i=1}^\infty a_if(x_i)+\sum_{i=1}^\infty b_if(y_i) &= \int_a^b f(x)\, dS_L\alpha+\int_a^b f(x)\, dS_R\alpha\\
       &= (DS)\int_a^b f(x)\, dS_L\alpha+(DS)\int_a^b f(x)\, dS_R\alpha\\
       &= \sum_{i=1}^\infty a_if(x_i+)+\sum_{i=1}^\infty b_if(y_i-).
   \end{align*}
   If we rearrange terms, we get~\eqref{eqn:RDSiffDS-2}.

\medskip

   Now suppose that $f\in RDS_\alpha[a,b]$, $f(x_i+)$ and $f(y_i-)$ exist, and~\eqref{eqn:RDSiffDS-2} holds.     Then, by Theorem~\ref{thm:ds-lebesgue-crit}, $f\in DS_{S_L\alpha}[a,b]$ and $f\in DS_{S_R\alpha}[a,b]$. Moreover, since $f\in RDS_\alpha[a,b]$, by Lemma \ref{lem:RDSGalpha}, $f\in RDS_{G\alpha}[a,b]$ and so by Theorem \ref{thm:RDSiffDSintegrator}, $f\in DS_{G\alpha}[a,b]$. Therefore, by \cite[Theorem 4.30c]{TSI}, $f\in DS_\alpha[a,b]$. Finally, if we argue exactly as we did at the end of the proof of Theorem~\ref{thm:RDSiffDSsuffcond}, we have that~\eqref{eqn:RDSiffDS-1} holds.
\end{proof}

The hypothesis~\eqref{eqn:RDSiffDS-2} in Theorem \ref{thm: RDSiffDScharacterization} is often not easy to check.  However, as a corollary  we give a practical  sufficient condition that is weaker than that in Theorem~\ref{thm:RDSiffDSsuffcond}.  In particular, if $\alpha$ is not discontinuous at the endpoints of $[a,b]$, then by Corollary~\ref{cor:WLOGleftcontinuity} we may assume that $\alpha$ is either left or right continuous, depending on $f$.

\begin{corollary}
    Given $\alpha \in I[a,b]$ and $f\in B[a,b]$, suppose for every common point of discontinuity, either:
    \begin{enumerate}
        \item $\alpha$ is right continuous and $f$ is left continuous, or

        \item $\alpha$ is left continuous and $f$ is right continuous.
    \end{enumerate} 
    Then $f\in RDS_\alpha[a,b]$ if and only if $f\in DS_\alpha[a,b]$.  Moreover, $$\int_a^b f(x)\, d\alpha = (DS)\int_a^b f(x)\, d\alpha.$$
\end{corollary}

\begin{proof}
First given our hypothesis on $f$ and $\alpha$, by Theorems~\ref{thm:RDSLebcrit} and~\ref{thm:ds-lebesgue-crit}, and arguing as we did in the proof of Theorem~\ref{thm:RDSiffDSsuffcond} we have that $f\in DS_\alpha[a,b]$ if and only if $f\in RDS_\alpha[a,b]$.

It remains to show that integrals are equal.  
    Let $\{(x_i, a_i)\}$ be the reduced saltus set for $S_L\alpha$ and $\{(y_i, b_i)\}$ be the reduced saltus set for $S_R\alpha$. We will show that~\eqref{eqn:RDSiffDS-2} holds by showing $f(x_i+)=f(x_i)$ and $f(y_i-)=f(y_i)$ for all $i$. If $x_i$ is a common point of discontinuity between $f$ and $\alpha$, then $\alpha$ is not right continuous at $x_i$ as it is in the saltus set for $S_L\alpha$. Thus, by assumption, $\alpha$ is left continuous and $f$ is right continuous; that is, $f(x_i+)=f(x_i)$. Otherwise, if $x_i$ is not a common point of discontinuity, then $f$ must be continuous at $x_i$ and again $f(x_i+)=f(x_i)$. A similar argument shows that $f(y_i-)=f(y_i)$ for all $i$.
\end{proof}

\bigskip

We now consider Riemann-type integral corresponding to the RDS-integral.  In our definition we will follow the approach used in~\cite[Sections~2.2 and~5.3]{TSI}.  We will reinterpret the Riemann sums used to define the RS-integral (see Definition~\ref{defn:rs-integral}) as the integral of a step function from a restricted class we call Riemann step functions  However, for our definition we will use the RDS-integral of the step function, resulting in a more generalized Riemann sum.

\begin{definition}\label{def:taggedstep}
    A tagged partition $\mathcal{P}^*$ of $[a,b]$ is a partition $\mathcal{P}=\{x_i\}_{i=0}^n$ of $[a,b]$ together with a collection of sample points $\{x_i^*\}_{i=0}^n$ where $x_i^* \in \overline{I}_i.$ Given $f\in B[a,b]$ and a tagged partition $\mathcal{P}^*$ of $[a,b]$, we say $r\in S[a,b]$ is a Riemann step function of $f$ with respect to $\mathcal{P}^*$ if $r(x)=f(x_i^*)$ for $x\in I_i$ and $f(x_i)=f(x_i)$ for each $1\le i\le n$. The collection of all Riemann step functions of $f$ defined with respect to all tagged partitions $\mathcal{P}^*$ is denoted by $R^*(f,\mathcal{P})$.
\end{definition}

\begin{definition}\label{def:RRSint}
    Given $\alpha\in BV[a,b]$ and $f\in B[a,b]$, we say $f$ is modified Riemann-Stieltjes integrable on $[a,b]$ with respect to $\alpha$ if there exists $A\in \mathbb{R}$ such that for every $\epsilon>0$ there exists $\delta>0$ so that for any partition $\mathcal P$ with $|\mathcal{P}|<\delta$ and any $r\in R^*(f, \mathcal P)$, 
    $$\left|\int_a^b r(x)\, d \alpha - A\right|<\epsilon.$$ 
    In this case we define the value of the modified Riemann-Stieltjes integral of $f$ with respect to $\alpha$ by 
    $$(mRS)\int_a^b f(x)\, d \alpha = A$$ 
    and the collection of all such $f$ by $mRS_\alpha[a,b]$.
\end{definition}

\begin{remark}
    It might seem more natural to refer to this integral as the Ross-Riemann-Stieltjes integral or the RRS-integral.   However, we have reserved that notation for Riemann-type integral introduced by Ross:  see Definition~\ref{def:RRSintNewMesh} below.
\end{remark}

However, as is the case with the RS and DS-integrals, this definition does not agree with the definition of the RDS-integral, as the set of integrable functions is smaller.  In fact, we can use the same example as for those integrals:  see~\cite[Example~5.30]{TSI}.

\begin{example}\label{example:RRSisnotRDS}
    Given $\alpha=H_1\in BV[-1,1]$, there exists a function $f$, namely $H_1$, where $f\in RDS_\alpha[-1,1]$ while $f\not\in mRS_\alpha[-1,1]$.
\end{example}
\begin{proof}
    Since $H_1\in S[-1,1]$, we have by Proposition \ref{prop:stepagree} that $H_1\in RDS_\alpha[-1,1].$  To show that $H_1$ is not mRS-integrable, fix $A\in \mathbb{R}$ and let $\epsilon=\frac{1}{2}$. Fix $\delta>0$ and let $\mathcal{P}=\{x_i\}_{i=0}^n$ be a partition such that $|\mathcal{P}|<\delta$ and $0$ is not a partition point. Then, there exists $1\le i_0\le n$ where $0\in I_{i_0}$. Thus, $\mu_\alpha(I_{i_0})=1$ while $\mu_\alpha(I_i)=0$ for all $i\neq i_0$. Let $r,s\in R^*(f, \mathcal{P})$ such that $x_{i_0}^*$ lies in $(0, x_{i_0})$ for $r$ and $x_{i_0}^*$ lies in $(x_{i_0-1}, 0)$ for $s$. Hence, $r(x)=1$ and $s(x)=0$ for $x\in I_{i_0}.$ Therefore, 
    $$\int_{-1}^1 r(x)\, d \alpha = \sum_{i=0}^n r(x_i)\mu_\alpha(\{x_i\})+\sum_{i=1}^n r(x_i^*)\mu_\alpha(I_i) = r(x_{i_0}^*)=1,$$
   and the same argument shows that
    $$\int_{-1}^1 s(x)\, d \alpha = 0.$$ 
    Thus, $$1= \left|\int_{-1}^1 r(x)\, d \alpha-\int_{-1}^1 s(x)\, d \alpha\right|
    \le\left|\int_{-1}^1 r(x)\, d \alpha-A \right|+\left|\int_{-1}^1 s(x)\, d \alpha-A \right|.$$ 
    Hence, one of the terms on the right is at least $\epsilon=\frac{1}{2}$ and so we have the desired step function in $R^*(f, \mathcal{P}).$ Therefore, $H_1\not\in RS_\alpha[-1,1]$.
\end{proof}

\begin{remark}
    We leave it as an open problem to determine if the mRS-integral and RDS-integrals agree whenever the former is defined.  Moreover, there is also a modified definition for which the same questions can be asked.  Since in Definition~\ref{def:taggedstep} the sample point can be taken as one of the endpoints of the partition interval, we might think of this as an "exterior" modified RS-integral.  It is possible to define an "interior" modified RS-integral, by restricting the sample points to be in the interior of the partition interval.   We leave it as on open problem to define and characterize the interior modified RS-integral and determine its relationship to the exterior version we define and to the RDS-integral.  Similar questions exist for the RS and DS integrals:  see \cite[Theorem~5.31, Exercises~5.20,~5.29]{TSI}.
\end{remark}

\medskip

As Example~\ref{example:RRSisnotRDS} makes clear, the problem arises because a common point of discontinuity can be an interior point of a partition interval, no matter how small the mesh size.  As we noted in the introduction, to avoid this problem for the RS-integral, Pollard introduced the notion of "$\sigma$-convergence" which leads to a modified definition of the limit in the RS-integral.  Here we show that the same technique can be used to define a Ross-Pollard-Stieltjes integral that is equivalent to the RDS-integral.  

\begin{definition}\label{def:Pollardint}
    Given $\alpha\in BV[a,b]$ and $f\in B[a,b]$, we say $f$ is Ross-Pollard-Stieltjes integrable on $[a,b]$ with respect to $\alpha$ if there exists $A\in \mathbb{R}$ such that for every $\epsilon>0$, there exists a partition $\mathcal{P}_\epsilon$ so that for any refinement $\mathcal Q$ and any $r\in R^*(f, \mathcal Q)$, $$\left|\int_a^b r(x)\, d \alpha - A\right|<\epsilon.$$ In this case we define the value of the Ross-Pollard-Stieltjes integral of $f$ with respect to $\alpha$ by $$(RPS)\int_a^b f(x)\, d \alpha = A$$ and the collection of all such $f$ by $RPS_\alpha[a,b]$.
\end{definition}

For brevity, we will only prove that the RPS-integral and the RDS-integral are equivalent for  $\alpha\in I[a,b]$. To prove the general case when  $\alpha\in BV[a,b]$ we would need to prove that the RPS-integral satisfied a canonical decomposition of the integrator analogous to Theorem~\ref{thm:pnalpha} for the RDS integral.  Such a result is true for the RS-integral:  see~\cite[Lemma~5.38]{TSI}.

\begin{theorem}\label{thm:RPSisRDS}
    Given $\alpha\in I[a,b]$ and $f\in B[a,b]$, $f\in RDS_\alpha[a,b]$ if and only if $f\in RPS_\alpha[a,b]$. In this case, 
    \begin{equation} \label{eqn:RPSisRDS-1}
    \int_a^b f(x)\, d \alpha = (RPS)\int_a^b f(x)\, d \alpha.
    \end{equation}
\end{theorem}

\begin{proof} 
    If $\mu_\alpha([a,b])=0$, then $\alpha$ is constant and so $f$ is integrable in both definitions with common value is $0$. Indeed, by Lemma \ref{lem:constint}, we have this for the RDS-integral. To see this for the RPS-integral, fix $\epsilon>0$ and let $\mathcal{P}_\epsilon$ be the trivial partition. Then, for any refinement $\mathcal{Q}$ and $r\in R^*(f, \mathcal{Q})$, $$\left|\int_a^b r(x)\, d \alpha\right|=0<\epsilon.$$  Hence, without loss of generality, $\mu_\alpha([a,b])>0$. 
    
    Suppose first that $f\in RDS_\alpha[a,b]$. Let 
    $$A=\int_a^b f(x)\, d \alpha.$$ 
    Fix $\epsilon>0$. By Definition \ref{def:RDSint}, there exists $u,\,v\in S[a,b]$ which bracket $f$ and $$\int_a^b u(x)\, d \alpha <A+\epsilon, \quad \int_a^b v(x)\, d \alpha >A-\epsilon.$$ Without loss of generality, suppose $u,\,v$ are step functions with respect to a common partition $\mathcal{P}$. Let $\mathcal{P}_\epsilon = \mathcal{P}$. Let $\mathcal{Q}$ be any refinement of $\mathcal{P}$ and fix any $r\in R^*(f, \mathcal{Q})$. Since $u,\,v$ are also step functions with respect to $\mathcal{Q}$ and bracket $f$ we have that $v(x)\le r(x)\le u(x)$ for all $x\in [a,b]$. Thus, by Theorem \ref{thm:stepmon}, 
    $$A-\epsilon<\int_a^bv(x)\, d \alpha\le \int_a^b r(x)\, d \alpha \le \int_a^bu(x)\, d \alpha < A+\epsilon.$$ 
    Since $\epsilon>0$ is arbitrary, by Definition \ref{def:Pollardint}, $f\in RPS_\alpha[a,b]$ and \eqref{eqn:RPSisRDS-1} holds.

\medskip

    Now suppose $f\in RPS_\alpha[a,b]$ and let 
    $$A=(RPS)\int_a^b f(x)\, d \alpha.$$ %
    Fix $\epsilon>0$. By Definition \ref{def:Pollardint}, there exists a partition $\mathcal{P}_\epsilon=\{x_i\}_{i=0}^n$ such that for any refinement $\mathcal{Q}$ of $\mathcal{P}_\epsilon$ and $r\in R^*(f, \mathcal{Q})$, 
    $$\left|\int_a^b r(x)\, d \alpha -A\right|<\frac{\epsilon}{4}.$$ 
    It will actually be enough to consider $\mathcal{Q}=\mathcal{P}_\epsilon$. Denote the partition intervals of $\mathcal{P}_\epsilon$ by $I_i$. Let $c_i= \sup_{x\in \overline{I_i}} f(x)$ and $d_i=\inf_{x\in \overline{I_i}}f(x)$. Since $f$ is bounded, these are well-defined. Note that there exist $y_i,z_i\in \overline{I_i}$ such that %
    $$f(y_i)+\frac{\epsilon}{4\cdot \mu_{\alpha}([a,b])2^i}>c_i,\quad f(z_i)-\frac{\epsilon}{4\cdot \mu_\alpha([a,b])2^i}<d_i.$$ 
    Define
    $$u(x)=\begin{cases}
        c_i, & x\in I_i,\\ f(x),& x=x_i,
    \end{cases} \quad v(x)=\begin{cases}
        d_i, & x\in I_i,\\ f(x), & x=x_i.
    \end{cases}$$ 
    Observe that $u,\,v$ bracket $f$. Now define
    $$r_u(x)=\begin{cases}
        f(y_i), & x\in I_i,\\ f(x), & x=x_i,
    \end{cases} \quad r_v(x)=\begin{cases}
        f(z_i), & x\in I_i,\\ f(x), & x=x_i.
    \end{cases}$$ 
    Notice that $r_u, r_v\in R^*(f,\mathcal{P}_\epsilon)$ and 
    \begin{align*}
        \int_a^b u(x)\, d \alpha &= \sum_{i=0}^nu(x_i)\mu_\alpha(\{x_i\})+\sum_{i=1}^nc_i \mu_\alpha(I_i)\\
        &< \sum_{i=0}^nr_u(x_i)\mu_\alpha(\{x_i\})+\sum_{i=1}^nf(y_i)\mu_\alpha(I_i)+\sum_{i=1}^n \frac{\epsilon}{4\cdot \mu_{\alpha}([a,b])2^i}\mu_\alpha(I_i)\\
        &\le \int_a^b r_u(x)\, d \alpha + \sum_{i=1}^n \frac{\epsilon}{4\cdot 2^i}\\
        &< \int_a^b r_u(x)\, d \alpha +\frac{\epsilon}{4}.
    \end{align*} 
    Similarly, 
    $$\int_a^b v(x)\, d \alpha > \int_a^b r_v(x)\, d \alpha -\frac{\epsilon}{4}.$$ 
    Therefore, 
    \begin{multline*}
         \int_a^b u(x)-v(x)\, d \alpha < \int_a^b r_u(x)\, d \alpha - \int_a^b r_v(x)\, d \alpha +\frac{\epsilon}{2}\\
         =\left[\int_a^b r_u(x)\, d \alpha-A\right] +\left[A- \int_a^b r_v(x)\, d \alpha\right] +\frac{\epsilon}{2}<\epsilon.
     \end{multline*} 
     Hence, by Theorem \ref{thm:RDScrit}, $f\in RDS_\alpha [a,b]$. By the previous argument, \eqref{eqn:RPSisRDS-1} holds.
    \end{proof}

\medskip

Ross \cite[Definition~35.24,Theorem~35.25]{Ross} gave an alternative definition of a Riemann-type integral associated to the RDS-integral.  His idea was to replace the mesh size of a partition with a mesh size defined using the $\alpha$-lengths of the interval.

\begin{definition}
    Given $\alpha\in BV[a,b]$, the $\alpha$-mesh size of a partition $\mathcal{P}=\{x_i\}_{i=0}^n$ is 
    $$|\mathcal{P}|_\alpha =\max_i\mu_{\alpha}(I_i).$$
\end{definition}

Using this he defined what we refer to as the Ross-Riemann-Stieltjes integral, and showed that it was equivalent to the RDS-integral.  We refer the reader to \cite{Ross} for the proof.

\begin{definition}\label{def:RRSintNewMesh}
    Given $\alpha\in BV[a,b]$ and $f\in B[a,b]$, we say $f$ is Ross-Riemann-Stieltjes integrable on $[a,b]$ with respect to $\alpha$ if there exists $A\in \mathbb{R}$ such that for every $\epsilon>0$ there exists $\delta>0$ so that for any partition $\mathcal P$ with $|\mathcal{P}|_\alpha <\delta$ and any $r\in R^*(f, \mathcal P)$, 
    $$\left|\int_a^b r(x)\, d \alpha - A\right|<\epsilon.$$ 
    We define the value of the Ross-Riemann-Stieltjes integral of $f$ with respect to $\alpha$ by $$(RRS)\int_a^b f(x)\, d \alpha = A$$ 
    and denote the collection of all such $f$ by $RRS_\alpha[a,b]$.
\end{definition}

\begin{theorem}\label{thm:RPSNewisRDS}
    Given $\alpha\in I[a,b]$ and $f\in B[a,b]$, $f\in RDS_\alpha[a,b]$ if and only if $f\in RRS_\alpha[a,b]$. In this case, 
    $$\int_a^b f(x)\, d \alpha = (RRS)\int_a^b f(x)\, d \alpha.$$
\end{theorem}

\begin{remark}
In~\cite{Ross}, Ross only defined the $\alpha$-mesh size of a partition, and the RRS-integral for increasing $\alpha$.  It is an open problem to extend his definition to $\alpha\in BV[a,b]$ and determine whether this more general definition is equivalent to the RDS-integral.  
\end{remark}

\bigskip

Finally, we will show that the RDS-integral agrees with the Lebesgue-Stieltjes integral whenever the former is defined. We use the standard definition for the Lebesgue-Stieltjes integral which can be found, for example, in Kamke~\cite{Kamke_1956},  Royden~\cite{MR1013117}, or Hewitt and Stromberg~\cite{MR0188387}. (See also Carter and van Brunt~\cite{LebesgueStieltjesCarter},which uses an equivalent but nonstandard definition.)   Recall that  $\mu_\alpha$, as given in Definition~\ref{def:aleng}, is a premeasure on $[a,b]$.  Therefore, by the Caratheodory Extension Theorem~\cite[Section~12.2]{MR1013117}, there exists a regular Borel measure $\mu^\alpha$ that agrees with $\mu_\alpha$ on intervals.   We will denote  the Lebesgue-Stieltjes integral associated to $\alpha$ by
$$\int_{[a,b]} f(x)\, d \mu^\alpha.$$ 

\begin{theorem}\label{thm:RDSLebEquiv}
    Given $\alpha\in BV[a,b]$ and $f\in B[a,b]$, if $f\in RDS_{\alpha}[a,b]$, then $f$ is Lebesgue-Stieltjes integrable on $[a,b]$ and 
\begin{equation} \label{eqn:RDSL-1}
\int_a^b f(x)\,d\alpha = \int_{[a,b]} f(x)\, d \mu^\alpha.
\end{equation}
\end{theorem}

To prove Theorem~\ref{thm:RDSLebEquiv} we need one lemma whose proof is given below.
\begin{lemma}\label{lem:bestfitlimit}
    Given $\alpha\in I[a,b]$, if $f\in RDS_\alpha[a,b]$, then there exist two sequences of step functions $\{u_n\}_{n=1}^\infty$ and $\{v_n\}_{n=1}^\infty$ such that $v_n(x)\le v_{n+1}(x)\le f(x)\le u_{n+1}(x)\le u_n(x)$ for all $n\in \mathbb{N}$ and $x\in [a,b]$ and 
    \begin{equation} \label{eqn:bestfitlimit1}
\lim_{n\to \infty}\int_a^b u_n\, d\alpha = \lim_{n\to \infty}\int_a^b v_n\, d\alpha = \int_a^b f(x)\, d\alpha.
\end{equation}
\end{lemma}

\begin{proof}[Proof of Theorem \ref{thm:RDSLebEquiv}]
    Since the premeasure of $\mu^\alpha$ is $\mu_\alpha$, they agree on intervals, and so we immediately have that step functions are Lebesgue-Stieltjes integrable and their RDS and LS-integrals agree.
    
    First, we consider the case $\alpha\in I[a,b]$ and $f$ is nonnegative. By Lemma \ref{lem:bestfitlimit}, there exists a decreasing sequence of step functions  $\{u_n\}_{n=1}^\infty$ and an increasing sequence $\{v_n\}_{n=1}^\infty$ such that $u_n$ and $v_n$ bracket $f$ for all $n$ and \eqref{eqn:bestfitlimit1} holds.    Since these sequences are monotone and bounded, we can define their pointwise limits
    $$u(x) = \inf_n u_n(x) \quad \text{and} \quad  v(x)=\sup_n v_n(x).$$
    Note that $v(x)\le f(x)\le u(x)$ and $u$ and $v$ are bounded on $[a,b]$. Thus, by the Dominated Convergence Theorem and the equality of the step function integrals,
\begin{multline*}
    \int_{[a,b]} v\, d\mu^\alpha = \lim_{n\to \infty} \int_{[a,b]} v_n\, d\mu^\alpha =  \lim_{n\to \infty} \int_a^b v_n\, d\alpha
    = \int_a^b f(x)\, d\alpha\\ = \lim_{n\to \infty}\int_a^b u_n\, d\alpha = \lim_{n\to \infty}\int_{[a,b]} u_n\, d\mu^\alpha =\int_{[a,b]} u\, d\mu^\alpha.
\end{multline*} 
Since $u(x)-v(x)\ge 0$ and their integrals agree,  we have $u=v$ $\mu^\alpha$-a.e. Hence,  $f=u$ $\mu^\alpha$-a.e., and so it is Lebesgue-Stieltjes integrable and~\eqref{eqn:RDSL-1} holds.

  We next consider the general case. By Theorem \ref{thm:pnalpha}, $f\in RDS_{P\alpha}[a,b]$ and $f\in RDS_{N\alpha}[a,b]$. By Lemma \ref{lem:fplus}, $f^\pm \in RDS_{P\alpha}[a,b]$ and $f^\pm\in RDS_{N\alpha}[a,b]$.  But by the previous case, $f^\pm$ are LS-integrable with respect to $\mu^{P\alpha}$ and $\mu^{N\alpha}$ and their LS and RDS-integrals agree. Further, since $\alpha = P\alpha - N\alpha + \alpha(a)$, by \cite[Theorem 9.1.0]{Kamke_1956}, $\mu^\alpha = \mu^{P\alpha} -\mu^{N\alpha}$. Therefore, we can estimate as follows:
  \begin{align*}
       & \int_a^b f(x)\, d \alpha \\ 
       &\qquad  = \int_a^b f^+(x)\, d P\alpha -\int_a^b f^+(x)\, d N\alpha
        - \int_a^b f^-(x)\, d P\alpha + \int_a^b f^-(x)\, d N\alpha\\
        &\qquad  = \int_{[a,b]}f^+(x)\, d\mu^{P\alpha}-\int_{[a,b]}f^+(x)\, d\mu^{N\alpha}
        - \int_{[a,b]}f^-(x)\, d\mu^{P\alpha} + \int_{[a,b]}f^-(x)\, d\mu^{N\alpha}\\
        &\qquad  =\int_{[a,b]}f(x)\, d\mu^{P\alpha}-\int_{[a,b]}f(x)\, d\mu^{N\alpha}\\
        &\qquad  = \int_{[a,b]}f(x)\, d\mu^\alpha.
    \end{align*}

\end{proof}

\begin{proof}[Proof of Lemma \ref{lem:bestfitlimit}]
    We will construct the decreasing sequence $\{u_n\}_{n=1}^\infty$; the construction of the increasing sequence $\{v_n\}_{n=1}^\infty$ is nearly identical.  By Definition \ref{def:RDSint}, for each $n\in \mathbb{N}$, there exists a step function $u^\prime_n$, defined with respect to a partition $\mathcal{P}_n$, such that $u^\prime_n(x)\ge f(x)$ for all $x\in [a,b]$ and 
    $$\int_a^b u^\prime_n(x)\, d\alpha-\int_a^b f(x)\, d\alpha<\frac{1}{n}.$$ 
    Further, we may assume without loss of generality that for each $n\in \mathbb{N}$, $\mathcal{P}_{n+1}$ is a refinement of $\mathcal{P}_n$. Let $u_n$ be the best fit step function of $f$ with respect to $\mathcal{P}_n$. Since $\mathcal{P}_{n+1}$ is a refinement of $\mathcal{P}_n$, $f(x)\le u_{n+1}(x) \leq u_n(x)$ for each $n\in \mathbb{N}$ and $x\in [a,b]$.    Finally, 
    $$0 \leq \int_a^b u_n(x)\ d\alpha-\int_a^b f(x)\ d\alpha
    \le \int_a^b u^\prime_n(x)\ d\alpha-\int_a^b f(x)\ d\alpha
    <\frac{1}{n}.$$ 
    Therefore, the limit for $\{u_n\}_{n=1}^\infty$ in \eqref{eqn:bestfitlimit1} holds.
\end{proof}

\bibliographystyle{plain}
\bibliography{ref}

\end{document}